\documentclass[english]{amsart}
\usepackage{amsmath}
\usepackage{amssymb}
\usepackage[all]{xy}
\usepackage{xypic,amsthm,amssymb,hyperref,amsmath,graphicx,stmaryrd,boxedminipage,mathrsfs,manfnt,todonotes,comment}
\usepackage[all]{xy}
\usepackage{sseq}

\newcommand{\C}{\mathbf{C}}
\newcommand{\Z}{\mathbf{Z}}
\newcommand{\Q}{\mathbf{Q}}

\newcommand{\sma}{\wedge}

\newcommand{\im}{\operatorname{Im}}

\newcommand{\id}{\operatorname{Id}}

\newcommand{\Spec}{\operatorname{Spec}}

\newcommand{\mc}{\mathcal}

\newcommand{\mbf}{\mathbf}

\newtheorem{thm}{Theorem}[section]
\newtheorem{lem}[thm]{Lemma}
\newtheorem{cor}[thm]{Corollary}
\newtheorem{prop}[thm]{Proposition}

\newtheorem{conj}[thm]{Conjecture}

\theoremstyle{definition}
\newtheorem{defn}[thm]{Definition}
\newtheorem{example}[thm]{Example}
\newtheorem{rmk}[thm]{Remark}
\newtheorem{nota}[thm]{Notation}

\begin{document}

\title{The $K$-Theory Spectrum of Varieties}
\author{Jonathan A. Campbell}
\email{j.campbell@math.vanderbilt.edu}
\address{\newline Vanderbilt University \newline Department of Mathematics \newline 1326 Stevenson Center \newline Nashville, TN 37240}

\begin{abstract}
We produce an $E_\infty$-ring spectrum $K(\mbf{Var}_{/k})$ whose components model the Grothendieck ring of varieties (over a field $k$) $K_0 (\mbf{Var}_{/k})$. This is acheived by slightly modifying Waldhausen categories and the Waldhausen $S_\bullet$-construction.  As an application, we produce liftings of various motivic measures to spectrum-level maps, including maps into Waldhausen's $K$-theory of spaces $A(\ast)$ and to $K(\Q)$. 
  \end{abstract}

\maketitle

\section{Introduction}

Let $k$ be a field. The Grothendieck ring of varieties $K_0(\mbf{Var}_{/k})$ is as a group defined to have generators the isomorphism classes $[X]$ where $X$ is a variety over $K$, and relations $[X-Y] + [Y] = [X]$ where $Y \hookrightarrow X$ is a closed inclusion. The multiplication is induced by Cartesian product of varieties. This ring is a fundamental object of study for algebraic geometers:  it is a universal home for Euler characteristics of varieties, called motivic measures,  as well as an easy version of motives. It has further deep ties to stable birational geometry, and a number of interesting statements in that field can be phrased in terms of the structure of $K_0 (\mbf{Var}_{/k})$ (see, e.g. \cite{liu_sebag, larsen_lunts}).  The Grothendieck ring of varieties also arises as the target for ``motivic integration'' \cite{looijenga}, a technique invented by Kontsevich for producing rational invariants of Calabi--Yau varieties. In his setup, the target for such an integral is a ring closely related to $K_0 (\mbf{Var}_{/k})$. In general, any ring homomorphism $K_0 (\mbf{Var}_{/k}) \to A$ can be used as a measure for motivic integration, hence the term motivic measure. 

The construction of motivic measures is a powerful technique for understanding the structure of $K_0 (\mbf{Var}_{/k})$, and a number of authors have constructed interesting ones. For example, in \cite{larsen_lunts} the authors construct a motivic measure $K_0 (\mbf{Var}_{/k}) \to \mbf{Z}[SB]$ where the latter denotes the free group ring on stable birational classes of varieties. Furthermore, they show that the kernel of that ring map is the ideal generated by the class of the affine line $[\mbf{A}^1_k]$. The same motivic measure was also used in \cite{larsen_lunts} to prove the irrationality of a certain motivic zeta function. Another slightly more exotic motivic measure was produced in \cite{bondal_larsen_lunts}: a ring map $K_0 (\mbf{Var}_{/k}) \to K_0 (\mbf{PT})$ where $\mbf{PT}$ is the category of small pre-triangulated categories.

Given a Grothendieck ring $K_0$, topologists and algebraic $K$-theorists have come to expect concomitant higher $K$-groups, $K_i$, that arise as homotopy groups of spaces or spectra. This is the case with the algebraic $K$-theory of rings \cite{bass,milnor,quillen}, and the algebraic $K$-theory of topological spaces \cite{waldhausen}. In the case of $K_0 (\mbf{Var}_{/k})$  Zakharevich \cite{zakharevich} showed, using her formalism of assemblers, there is indeed an underlying spectrum and used the result to prove a number of results relating to cut-paste conjectures. We will call the spectrum she defined $K(\mbf{Var}_{/k})$.

It is interesting to study the higher homotopy groups of $K(\mbf{Var}_{/k})$ and there are concrete reasons to believe the higher homotopy contains a great deal of geometric information. For example, Zakharevich \cite{zakharevich_annihilator} has used $\pi_1 K(\mbf{Var}_{/k})$ very effectively to study questions in birational geometry.  For other flavors of algebraic $K$-theory, the typical way to study higher $K$-theory is to produce maps from $K(\mbf{Var}_{/k})$ to target spectra with computable homotopy groups --- in our case such maps correspond to ``derived'' motivic measures. Unfortunately, assemblers are very difficult to define maps out of, and a different construction of $K(\mbf{Var}_{/k})$ is needed. This paper provides such a construction. We note now that it takes work to prove the equivalence of the models. The comparison will appear in future work of the author, Jesse Wolfson, and Inna Zakharevich. 

The standard way of defining higher algebraic $K$-theory begins with a category $\mc{E}$ with some notion of ``exact sequence'', for example Quillen's exact categories \cite{quillen} or triangulated categories. One then defines $K_0 (\mc{E})$ in the usual way by splitting exact sequences. Roughly, the higher $K$-groups are defined by using simplicial machinery to keep track of \textit{how} the sequences split. Waldhausen realized that in fact this type of machine works for a much less restrictive structure on the underlying category. One needs a zero object, ``cofibrations'', which are maps $X \to Y$ where one can define a quotient $Y/X$, and some mild categorical conditions on the existence of certain colimits \cite{waldhausen}. Granted this structure on a category $\mc{C}$ one can define a spectrum $K(\mc{C})$ by again using simplicial machinery to keep track of the ways in which $Y$ splits into $X$ and $Y/X$. In this case, the category together with the necessary structure is called a Waldhausen category, and the machinery is called the Waldhausen $S_\bullet$-construction.

One could hope to define a higher $K$-theory of varieties using such standard constructions. Unfortunately, there are immediate problems --- for example, $\mbf{Var}_{/k}$ cannot be a Waldhausen category since it has no zero object, nor does it have quotients or pushouts in general. However, these objections can be remedied, and we introduce a new formalism where a modified $S_\bullet$-construction can be run. The production of this modified $S_\bullet$-construction is the main point of this paper. 

First, the category $\mbf{Var}_{/k}$ has just enough pushouts: pushouts where both legs are closed inclusions exist. Also, ``quotients'' in our setting will be replaced by ``subtraction,'' $Y - X$ for closed inclusions $X \hookrightarrow Y$ of varieties -- it is these ``subtraction sequences'' $X \hookrightarrow Y \leftarrow Y - X$ that we will split. We also observe that a zero object is not actually needed, but an initial object is and the empty variety will work in this case. Proceeding in this way, we create a new formalism of SW-categories (for semi-Waldhausen or scissors-Waldhausen or subtractive-Waldhausen...) and a suitably modified Waldhausen $S_\bullet$-construction called the $\widetilde{S}_\bullet$-construction (see Section 3 for details).

For Waldhausen's $S_\bullet$-construction, the main theorem, and the theorem from which almost all $K$-theory theorems follow \cite{staffeldt} is the Additivity Theorem \cite[Thm. 1.4.2]{waldhausen}. The main theorem of this paper is the following analogue.

\begin{thm}
  Let $\mc{C}$ be an SW-category. Then the category $\mbf{Sub}(\mc{C})$ of subtraction diagrams $X \hookrightarrow Y \leftarrow Y-X$ can also be made into an SW-category. Furthermore
  \[
  \widetilde{S}_\bullet \mbf{Sub}(\mc{C}) \to \widetilde{S}_\bullet \mc{C} \times \widetilde{S}_\bullet \mc{C}
  \]
  is a weak equivalence of simplicial sets, with the map given by projecting onto to the first and last components of $X \hookrightarrow Y \leftarrow Y-X$. 
\end{thm}

With some work, this implies the following theorem.

\begin{thm}
For $\mc{C}$ an SW-category, $K(\mc{C})$ is an infinite loop space. 
\end{thm}

Since $\mbf{Var}_{/k}$ is an SW-category, we obtain a spectrum $K(\mbf{Var}_{/k})$. 

\begin{prop}
  The components of the spectrum $K(\mbf{Var}_{/k})$ coincide with the Grothendieck group of varieties:
  \[
  \pi_0 K(\mbf{Var}_{/k}) = K_0 (\mbf{Var}_{/k}).
  \]
\end{prop}

The category $\mbf{Var}_{/k}$ is endowed with a product, and we can descend this product to the spectrum level giving us an even stronger statement:

\begin{thm}
The cartesian product on varieties gives $K(\mbf{Var}_{/k})$ the structure of an $E_\infty$-ring spectrum. Furthermore, $\pi_0 K(\mbf{Var}_{/k})$ coincides with the Grothendieck ring of varieties. 
\end{thm}

Once the spectrum $K(\mbf{Var}_{/k})$ has been defined using a relative of well-studied machinery, we proceed to define maps in and out of $K(\mbf{Var}_{/k})$. These may be considered to be ``derived'' versions of motivic measures. The ability to do this is one of the main virtues of defining $K(\mbf{Var}_{/k})$ in this way. 

First, one can define a model for the unit map $S \to K(\mbf{Var}_{/k})$. Next, when $k$ is a finite field, a point-counting functor defines a map from $K(\mbf{Var}_{/k})$ to the sphere spectrum. One may also consider a complex variety as a topological space and relate this to Waldhausen's $K$-theory of spaces, $A(\ast)$ \cite{waldhausen}. Finally, one can define maps to $K(\mbf{Q})$ by using derived versions of the Euler characteristics. Summarizing, we have

\begin{thm}
  There are non-trivial spectrum maps (and explicit models for them)
  \begin{enumerate}
  \item $S \to K(\mbf{Var}_{/k})$ which on $\pi_0$  gives the map $\Z \to K_0 (\mbf{Var}_{/k})$ sending
    \[
    [n] \mapsto \underbrace{\Spec(k) \amalg \cdots \amalg \Spec(k)}_{n \ \text{times}}
    \]
  \item $K(\mbf{Var}_{/k}) \to S$, $k$ finite, which on $\pi_0$ gives the point-counting map $[X]  \mapsto \#X(k)$.
  \item $K(\mbf{Var}_{/\C}) \to A(\ast)$, which on $\pi_0$ is $[X] \mapsto \chi(X)$. 
  \item $K(\mbf{Var}_{/\C}) \to K(\operatorname{Ch}^{hb} (\Q))$ and $K(\mbf{Var}_{/\C}) \to K(\Q)$, where $\operatorname{Ch}^{hb}(\Q)$ denotes the category of homologically bounded chain complexes. On $\pi_0$ this is also $[X] \mapsto \chi(X)$. 
  \end{enumerate}
\end{thm}

There should be maps from $K(\mbf{Var}_{/k})$ into much ``larger'' and more interesting ring spectra. As a putative example of how to produce such a map, we consider the following. Instead of discarding information by simply counting points or taking cohomology, one could instead pass to derived categories, i.e assign a variety $X$ to its derived category $\mc{D}(X)$.  Done carefully, this procedure should product a functor from varieties to stable $\infty$-categories.  This would give us a conjectural map $K(\mbf{Var}_{/k}) \to K(\mc{C}\text{at}^{Ex}_\infty)$ where $\mc{C}\text{at}^{Ex}_\infty$ is the $\infty$-category of stable $\infty$-categories \cite{blumberg_gepner_tabuada, lurie}. A more concrete manifestion of this map is the following conjecture. 

\begin{conj}
  There is a map $K(\mbf{Var}_{/k}) \to K(K(S))$ or $K(K(k))$ of $E_\infty$-ring spectra. 
\end{conj}

\begin{rmk}
  The conjecture above would essentially supply a lift of of the Bondal--Larsen--Lunts motivic measure $K_0 (\mbf{Var}_{/k}) \to K_0 (\mbf{PT})$. 
\end{rmk}

There are many other possible ways of producing interesting motivic measures, and this will be the subject of future work. 

\subsection{Acknowledgements}

This paper grew out of a seminar conducted with Andrew Blumberg at UT-Austin in Fall of 2014. I thank Andrew for suggesting the seminar topic, many helpful conversations about this paper, and general supportive enthusiasm for the project. At various points Sean Keel, Jen Berg, and Ben Williams have answered very naive questions about algebraic geometry.  I am indebted to Jesse Wolfson for suggesting parts of the key definition \ref{w_exact}. Finally, I thank Inna Zakharevich for interest in the current work and encouragement. As is hopefully clear from the text, she was first to define the $K(\mbf{Var}_{/k})$ spectrum, and this paper represents an alternate approach to work she has done. Inna also pointed out to me that Torsten Ekedahl was apparently thinking of an approach to $K(\mbf{Var}_{/k})$ similar to the below at the time of his death \cite{ekedhal_overflow}. Denis-Charles Cisinski made very helpful comments on an earlier draft of this paper. The comments of an anonymous referree greatly improved the structure and exposition of this paper.

\section{Scheme-Theoretic Preliminaries}

In topological contexts, the construction of $K$-theory via Waldhausen categories \cite{waldhausen}, depends heavily on having certain categorical limits and colimits. We cannot take for granted the existence of all (or any) limits and colimits in the category of varieties. However, in this section we show that all of the limits and colimits that will be necessary do, in fact, exist and we collect a number of other useful results. The author first learned this material in \cite{schwede}, but the material exists in the Stacks Project \cite[Tag 07RS]{stacks-project} as well.

\begin{defn}
In what follows a \textbf{variety} will be a finite-type, separated scheme over an arbitrary base scheme $X$. 
\end{defn}

\begin{nota}
Throughout, closed immersions in both varieties and schemes will be denoted with a hooked arrow $Z \hookrightarrow X$. Similarly, an open immersion will be denoted by $Y \xrightarrow{\circ} X$. 
\end{nota}

We will need two results. 

\begin{thm}\cite[Thm. 3.3]{schwede}, \cite[Tag 07RS]{stacks-project}
Suppose $A, B$ are rings and suppose $I \subset A$ is an ideal, and that there exists a map $f: B \to A/I$. Consider the diagram
\[
\xymatrix{
Z = \Spec A/I \ar[d] \ar[r] & X = \Spec B \\
Y = \Spec A  & 
}
\]
Then, 
\begin{enumerate}
\item The pushout $X \amalg_Z Y$ exists and is affine
\item $Y \to X \amalg_Z Y$ is a closed immersion
\item both $X \to X \amalg_Z Y$ and $Y \to X \amalg_Z Y$ are morphisms of schemes. 
\end{enumerate}
\end{thm}

\begin{thm}\cite[Cor. 3.7]{schwede}, \cite[Tag 07RS]{stacks-project}\label{pushouts_exist}
Let $Z \hookrightarrow X$ and $Z \hookrightarrow Y$ be closed immersions of schemes. Then $X \amalg_Z Y$ exists. 
\end{thm}

Although not stated explicitly in Schwede, it is a consequence of the proof of Thm. \ref{pushouts_exist} that closed inclusions are preserved by cobase change:

\begin{cor}
In the situation of Thm.\ref{pushouts_exist}, $X \to X \amalg_Z Y$ and $Y \to X \amalg_Z Y$ are closed immersions. 
\end{cor}

However, we will need more. Since we will be working just in the category of varieties, we need that in fact pushouts exists in that category. 

\begin{prop}\label{pushout_varieties}
Let $Z \to X$, $Z \to Y$ be closed embeddings of varieties. Form the pushout $X \amalg_Z Y$ in the category of schemes. Then  $X \amalg_Z Y$ is a variety. 
\end{prop}

  In the category of varieties, an ``exact sequence'' will be a sequence
\[
X \hookrightarrow Y \xleftarrow{\circ} Y - X
\]
where the first map is a closed embedding. These will be the sequences we want to split. In order to view them as the input to a $K$-theory machine, however, we have to verify a number of categorical properties. In the rest of the section we collect these properties.

First, we define how to subtract schemes. 

\begin{defn}
  Let $i: Z \hookrightarrow X$ be a closed immersion. We define $X - Z$ as follows. The immersion $i$ determines a homeomorphism onto a closed subset $i(Z) \subset X$, which in turn determines an open subset $X - i(Z)$ of $X$. To view this as a scheme, we restrict the structure sheaf $\mc{O}_X$ to $X - i(Z)$. That is,
  \[
  X - Z = (X - i(Z), \mc{O}_X|_{X - i(Z)})
  \]
\end{defn}

\begin{rmk}
  This is a good time to remark on the functoriality of subtraction. It is clear that given a diagram
  \begin{equation}\label{functorial_subtraction_diagram}
  \xymatrix{
    W \ar@{^{(}->}[r] \ar@{^{(}->}[d] & X\ar@{^{(}->}[d] \\
    Z \ar@{^{(}->}[r] & Y
  }
  \end{equation}
  there need not be an induced map $X - W \to Y - Z$. Indeed, if $X = Y$ and $W$ is strictly contained in $Z$, then $X-W$ \textit{contains} $Y-Z$. This is fixed, however, if we require that the diagram be cartesian. On the level of sets, this corresponds to intersecting $X$ and $Z$ inside of $Y$. We then get a map $X - W \hookrightarrow Y- Z$.

  In the case (\ref{functorial_subtraction_diagram}) is cartesian, there is more we can do. We can extend it to a diagram
  \[
  \xymatrix{
    W  \ar@{^{(}->}[r]  \ar@{^{(}->}[d] & X \ar@{^{(}->}[d] & X - W \ar@{^{(}->}[d] \ar[l]_{\circ}\\
    Z  \ar@{^{(}->}[r] & Y & Y - Z \ar[l]_{\circ} & \\
    Z-W  \ar[u]_{\circ}\ar@{^{(}->}[r] & Y - X \ar[u]_{\circ} &  \square \ar[u]^{\circ}\ar[l]_{\circ}
  }
  \]
  where
  \[
  (Y-Z)-(X-W) = \square = (Y-X)-(Z-W). 
  \]
  and all maps along the bottom and right border are uniquely determined.
\end{rmk}

In general, however, this is a choice of subtraction; there are many other choices isomorphic to it. A better way to encode subtraction is the following:

\begin{defn}\label{subtraction_sequence_varieties}
We define the collection of maps $Z \hookrightarrow X \xleftarrow{\circ} Y$ such that the left map is a closed immersion, the right map is an open immersion and the underlying topological space of $X$ is the disjoint union of the underlying topological spaces of $Z$ and $Y$ to be the  \textbf{subtraction sequences}. 
\end{defn}

Note that with $X - Z$ defined as above, $Z \hookrightarrow X \xleftarrow{\circ} X-Z$ is a subtraction sequence. However, working with subtraction sequences allow for choices of isomorphic subtractions.

A property of these subtraction sequences is that they are closed under pullback:

\begin{prop}\label{subtractions_pullback}
  Subtraction sequences are closed under pullback: Given a subtraction sequence $Z \hookrightarrow X \xleftarrow{\circ} Y$ and a map $X'$, the bottom row in the diagram below is a subtraction sequence:
  \[
  \xymatrix{
    Z \ar@{^{(}->}[r]  & X  & Y \ar[l]^\circ\\
    X' \times_X Z \ar[u] \ar@{^{(}->}[r] & X'\ar[u] & X' \times_X Y \ar[l]^\circ \ar[u]
  }
  \]
\end{prop}
\begin{proof}
The statement is clearly true for the underlying topological spaces and open and closed immersions are both closed under pullback. 
\end{proof}

We record a useful corollary of Prop. \ref{subtractions_pullback}

\begin{cor}
If $i: X \hookrightarrow Y$ and $j: Y \hookrightarrow Z$ are closed immersions, then $Y - X \to Z - X$ is a closed immersion. 
\end{cor}

We restate the observation of Schwede (in \cite{schwede} above Lem. 3.8) that if $X$ and $Y$ are closed in an ambient scheme $W$ with intersection ideals $\mc{I}_X$ and $\mc{I}_Y$ respectively, and $Z$ is the scheme intersection, then $X \amalg_Z Y$ is cut out by $\mc{I}_X \cap \mc{I}_Y$ in $W$. 

\begin{prop}\cite{schwede}\label{pushout_product}
  Given a cocartesian diagram of varieties
  \[
  \xymatrix{
    Z \ar[r]\ar[d] & X\ar[d] \\
    Y \ar[r] & W
  }
  \]
  where all maps are cofibrations, the map $X \amalg_Z Y \to W$ is a cofibration. 
\end{prop}

\begin{rmk}\label{pushout_pullback}
We also note that cocartesian diagrams above are cartesian squares. 
\end{rmk}

Pushout also interacts in a controlled way with subtraction sequences:

\begin{prop}\label{var_sub_pushout}
  Given a diagram
  \[
  \xymatrix{
    X '  \ar@{^{(}->}[d] & W' \ar@{_{(}->}[l]\ar@{^{(}->}[r] \ar@{^{(}->}[d] & Y'\ar@{^{(}->}[d]\\
    X & W \ar@{^{(}->}[r] \ar@{_{(}->}[l] & Y\\
    X'' \ar[u]^\circ & W'' \ar@{^{(}->}[r] \ar@{_{(}->}[l] \ar[u]^\circ & Y'' \ar[u]^\circ
  }
  \]
  such that the columns are subtraction sequences, and both top squares are cartesian squares, the pushouts of the rows form a subtraction sequence
  \[
  Y' \amalg_{W'} X' \hookrightarrow Y \amalg_W X \xleftarrow{\circ} Y'' \amalg_{W''} X''
  \]
\end{prop}
\begin{proof}
  We examine the diagram
  \[
  \xymatrix@C=.3cm@R=.3cm{
    W' \ar[rr] \ar[dr] \ar[dd] & & Y' \ar[dr]\ar[dd] & \\
    & X' \ar[dd] \ar[rr] & & X' \amalg_{W'} Y' \ar[dd]\\
    W \ar[rr]\ar[dr] & & Y \ar[dr] \\
     & X \ar[rr] & & X \amalg_{W} Y
  }
  \]
  where the back and left faces are cartesian and all arrows except the rightmost are closed immersions.  It suffices to show that the rightmost arrow is a closed immersion.

  The left square is cartesian and the bottom square is cocartesian and so cartesian by Rmk. \ref{pushout_pullback}. Thus the composite square
  \[
  \xymatrix{
    W' \ar[d]\ar[r] & X'\ar[d]\\
    Y \ar[r] & X \amalg_W Y 
  }
  \]
  is cartesian. Thus, by Prop. \ref{pushout_product} the map $X' \amalg_{W'} Y \to X \amalg_W Y$ is a closed immersion. But, $X' \amalg_{W'} Y' \to X' \amalg_{W'} Y$ is a closed immersion as well, and closed immersions compose. 
\end{proof}

\section{The $K$-Theory Spectrum of Varieties}

Typically, the input for algebraic $K$-theory is a category imbued with some notion of cofiber sequence that $K$-theory then splits: for cofiber sequences $A \to B \to C$ in $\mc{C}$, we have the relation $[B] = [A] + [C]$ in $K_0 (\mc{C})$ with more subtle information encoded in higher $K$-groups. One of the more general constructions of algbraic $K$-theory is due to Waldhausen, and the categories on which Waldhausen's machine operates are, naturally, called Waldhausen categories \cite{waldhausen}. We review the construction of the $K$-theory of a Waldhausen category $\mc{C}$ below. 

For the category $\mbf{Var}_{/k}$, the axioms of a Waldhausen category are certainly not satisfied: the sequences we would like to split, sequences of the form $Z \hookrightarrow X \leftarrow X-Z$, do not even have morphisms in the appropriate directions. To circumvent this issue, we introduce the formalism of subtractive categories (Def. \ref{subtractive_category}) and a modified version of Waldhausen's construction of $K$-theory for such categories. 

Once the $K$-theory space is constructed, we can show that it is in fact an infinite loop space, or spectrum. This is done by proving the additivity theorem, a rigorous version of the statement that $K$-theory splits exact sequences. The key point is that subtraction sequences interact well enough with pushouts to allow adaptations of proofs of additivity (e.g. \cite{mccarthy}) to go through in the new context.

\subsection{Waldhausen Categories and the $S_\bullet$-construction}

We give a rapid review of Waldhausen categories and the construction of the $K$-theory spectrum for Waldhausen categories.

\begin{defn}\cite[Sect. 1.1]{waldhausen}
  A \textbf{Walhdausen category} $\mc{C}$ is a category with an initial and terminal object $\ast$,  equipped with two distinguished subcategories
  \begin{enumerate}
  \item cofibrations, denoted $\mbf{co}(\mc{C})$, with arrows in $\mbf{co}(\mc{C})$ denoted $\hookrightarrow$
  \item weak equivalences, denoted $\mbf{w}(\mbf{C})$, with arrows in $\mbf{w}(\mc{C})$ denoted $\xrightarrow{\sim}$.
  \end{enumerate}
  These are required to satisfy the following axioms
  \begin{itemize}
  \item The isomorphisms of $\mc{C}$ are in both $\mbf{co}(\mc{C})$ and $\mbf{w}(\mc{C})$
  \item For $C \in \mc{C}$, $\ast \to \mc{C}$ is a cofibration.
  \item Given a cofibration $C \hookrightarrow D$ and any arrow $C \to C'$, the pushout $C' \amalg_C D$ exists and furthermore $C' \to C' \amalg_C D$ is a cofibration
  \item Given a diagram
    \[
    \xymatrix{
      D \ar[d]^{\sim}  & C \ar[d]^{\sim}\ar[d]\ar@{_{(}->}[l]\ar[r] & E \ar[d]^{\sim}\\
      D' & C' \ar@{_{(}->}[l]\ar[r] & E'
    }
    \]
    the induced map $D \amalg_C E \to D' \amalg_{C'} E'$ is a weak equivalence. 
  \end{itemize}
\end{defn}

\begin{rmk}
Given a cofibration $C \hookrightarrow D$, one can form a pushout along $C \to \ast$ to form a quotient $D/C$. The resulting sequence $C \hookrightarrow D \to D/C$ is called a cofiber sequence. 
\end{rmk}

This is all the structure that is required to define $K$-theory. Note first, there is certainly a notion of Grothendieck group for a Waldhausen category, $\mc{C}$: it has generators the isomorphism classes $[C]$ with $C \in \mc{C}$ and relations $[C]+[D/C] = [D]$ for cofiber sequences $C \hookrightarrow D\to D/C$.

Before going on, we define the notion of functor between Waldhausen categories.

\begin{defn}
Let $\mc{C}$ and $\mc{D}$ be Waldhausen categories. Then a functor $F: \mc{C} \to \mc{D}$ is \textbf{exact} if it preserves zero objects, cofibrations, pushouts along cofibrations and weak equivalences. 
\end{defn}

To define higher $K$-groups, we need the $S_\bullet$ construction.

\begin{defn}
  Let $\mc{C}$ be a Waldhausen category. Let $\operatorname{Ar}[n]$ denote the arrow category: the objects are pairs $(i, j)$ with $0 \leq i \leq j \leq n$ and morphisms are $(i, j) \to (i', j')$ with $i \leq i'$ and $j \leq j'$. We define a category $S_n \mc{C}$ to be the full subcategory of functors $F: \operatorname{Ar}[n] \to \mc{C}$ such that
  \begin{enumerate}
  \item $F(i, i) = \ast$
  \item $F(i, j) \to F(i, k)$ is a cofibration for all $i \leq j \leq k$
  \item The square
    \[
    \xymatrix{
      F(i, j) \ar[d]\ar[r] & F(i, k)\ar[d]\\
      F(j, j)=\ast \ar[r] & F(j, k)
      }
    \]
    is cocartesian for all $i \leq j \leq k$. 
  \end{enumerate}

  The categories $S_n \mc{C}$ assemble into a simplicial category, which we will denote $S_\bullet \mc{C}$. The simplicial face maps $d_i: S_n \mc{C} \to S_{n-1} \mc{C}$ are given by deleting the $i$th row and $i$th column from a diagram in $S_n \mc{C}$. The degeneracies are given by inserting identity maps in the appropriate places. 
\end{defn}

We may now define the $K$-theory space.

\begin{defn}
  Let $w S_n \mc{C}$ be the subcategory of $S_n \mc{C}$ where morphisms are given by level-wise weak equivalences. We may form the simplicial category $w S_\bullet \mc{C}$ and take the level-wise nerve $Nw S_\bullet \mc{C}$, which we denote $w_\bullet S_\bullet \mc{C}$. The level-wise nerve $w_\bullet S_\bullet \mc{C}$ is a bisimplicial set. We define the \textbf{algebraic $K$-theory space of $\mc{C}$} to be
  \[
  K(\mc{C}):= \Omega |w_\bullet S_\bullet \mc{C}|
  \]
  where $|-|$ denotes the realization of a bisimplicial set. 
\end{defn}

Walhausen shows that in fact $K(\mc{C})$ is an infinite loop space. The crucial step is the additivity theorem:

\begin{thm}\cite[Prop 1.3.2]{waldhausen}
  Let $\mc{C}$ be a Waldhausen category and let $\mc{E}$ be the category whose objects are cofibration sequences $A \to B \to C$ in $\mc{C}$ and level-wise morphisms. Then $\mc{E}$ can be given the structure of a Waldhausen category.

  Furthermore, there are functors $s, q: \mc{E} \to \mc{C} \times \mc{C}$ given by taking $(A \to B \to C)$ to $A$ and $C$ respectively, and these induce an equivalence of simplicial sets
  \[
  w S_\bullet \mc{E} \xrightarrow{(s,q)} w S_\bullet \mc{C} \times w S_\bullet \mc{C}
  \]
\end{thm}

\subsection{Subtractive Categories}

In this section, we define ``categories with subtraction'', the minimal categorical input needed to define an analogue of the Waldhausen $S_\bullet$-construction. We go on to define ``subtractive categories'', which are more restrictive and which provide just enough structure to mimic standard proofs of the additivity theorem.  

\begin{defn}\label{cat_w_cofibs}
A \textbf{category with subtraction} is a category $\mc{C}$, equipped with a subcategory of cofibrations, $\mbf{co}(\mc{C})$ and a subcategory of fibrations, $\mbf{fib}(\mc{C})$. The arrows of $\mbf{co}(\mc{C})$ will be denoted by ``$\hookrightarrow$'' and those of $\mbf{fib}(\mc{C})$ will be denoted by $\xrightarrow{\circ}$. The following axioms must hold:
\begin{enumerate}
\item There is an initial object, typically referred to as the empty object, $\emptyset$. 
\item Isomorphisms are cofibrations and fibrations
\item (\textbf{pullbacks}) Pullbacks along cofibrations and fibrations exist, and satisfy base-change. 
\item   There is a notion of subtraction: that is, there is a collection of \textbf{subtraction sequences} $\{Z \hookrightarrow X \leftarrow Y\}$ which are required to satisfy the following axioms
  \begin{enumerate}
  \item $A \to A \amalg B \leftarrow B$ is a subtraction sequence
  \item Every cofibration $Z \hookrightarrow X$ participates in a subtraction sequence $Z \hookrightarrow X \leftarrow Y$ where $Y$ is unique up to unique isomorphism. The same statement holds for fibrations. We will informally denote $Y$ by $X-Z$. 
  \item Subtraction is functorial in fiber squares. Given a fiber square where all arrows are cofibrations, we can form the diagram below where all of the rows and columns are subtraction sequences
    \[
    \xymatrix{
    W  \ar@{^{(}->}[r]  \ar@{^{(}->}[d] & X \ar@{^{(}->}[d] & X - W \ar@{^{(}->}[d] \ar[l]_{\circ}\\
    Z  \ar@{^{(}->}[r] & Y & Y - Z \ar[l]_{\circ} & \\
    Z-W  \ar[u]_{\circ}\ar@{^{(}->}[r] & Y - X \ar[u]_{\circ} &  \square \ar[u]^{\circ}\ar[l]_{\circ}
  }
    \]
    where here $\square$ dentoes what would informally be called $(Y-Z)-(X-W)$ or $(Y-X)-(Z-W)$. In this diagram, we require that the arrows along the bottom and right of the diagram be uniquely determined and that the bottom right square be cartesian.

    The dual statement is required for fibrations.
    
  \item Subtraction is respected by base change. That is, given a subtraction sequence $Z \hookrightarrow X \xleftarrow{\circ} Y$ and a map $W \to X$ we can form the diagram where both squares are cartesian:
    \[
    \xymatrix{
      Z \ar@{^{(}->}[r] & X & Y \ar[l]_{\circ}\\
      Z \times_W X \ar@{^{(}->}[r]\ar[u] & W \ar[u]&  W \times_X Y \ar[l]_{\circ}\ar[u]
    }
    \]
    The bottom row is required to be a subtraction sequence. 
    \end{enumerate}
  \end{enumerate}
\end{defn}

\begin{rmk}
The definition of subtraction may seem somewhat odd, given that we don't specify what, exactly, the subtraction should be, only that it exist. It is, however, all that we need for any of the arguments below. It also leaves room for ``relative'' subtraction sequences, where given $\mc{C} \to \mc{D}$ an inclusion of categories with subtraction, we could define a new subtraction structure on $\mc{D}$ by declaring $Z \hookrightarrow X \leftarrow Y$ to be a subtraction sequence if $X - Z = Y \amalg C$ in the old structure. 
\end{rmk}

\begin{rmk}
The axiom for the functoriality of subtraction is necessitated by the fact that subtraction does not satisfy any good categorical properties: the intuitively suggested properties of subtraction must be inserted by fiat. 
\end{rmk}

\begin{rmk}
The definition bears a strong resemblance to the definition of exact categories in \cite{quillen}. 
\end{rmk}

There are a large number of examples of category with subtraction. The most important for us will be the following. 

\begin{example}\label{varieties_are_categories_with_subtraction}
  Let $X$ be a scheme. Then $\mbf{Sch}_{/X}$ or $\mbf{Var}_{/X}$ with cofibrations the closed inclusions are categories with subtraction. The cofibrations will be closed immersions and the subtraction sequences will be subtraction sequences of schemes or varieties Defn.\ref{subtraction_sequence_varieties}. It is clear that cofibrations and fibrations satisfy base change and that subtraction sequences also satisfy base change, thus $\mbf{Sch}_{/X}$ and $\mbf{Var}_{/X}$ are categories with subtraction.

\end{example}

\begin{example}
  Smooth schemes $\mbf{Sch}^{\text{sm}}_{/X}$ with cofibrations the closed inclusions or open inclusions.
\end{example}

Categories with subtraction are useful, but we will need a refinement of them in order to prove additivity.

\begin{defn}\label{subtractive_category}
  A \textbf{subtractive category}, $\mc{C}$, is a category with subtraction such that
  \begin{enumerate}
  \item (\textbf{pushouts}) The pushout of of a diagram where both legs are cofibrations exist and satisfy base change. Furthermore, cocartesian diagrams of this form are required to be cartesian
  \item (\textbf{pushout products}) In a cartesian square
    \[
    \xymatrix{
      W \ar[r]\ar[d] & X\ar[d] \\
      Y \ar[r] & Z
    }
    \]
    where all arrows are cofibrations, the map $X \amalg_W Y \to Z$ is a cofibration. 
  \item (\textbf{subtraction and pushouts}) Given a diagram
     \[
  \xymatrix{
    X '  \ar@{^{(}->}[d] & W' \ar@{_{(}->}[l]\ar@{^{(}->}[r] \ar@{^{(}->}[d] & Y'\ar@{^{(}->}[d]\\
    X & W \ar@{^{(}->}[r] \ar@{_{(}->}[l] & Y\\
    X'' \ar[u]^{\circ} & W'' \ar@{^{(}->}[r] \ar@{_{(}->}[l] \ar[u]^{\circ} & Y'' \ar[u]^{\circ}
  }
  \]
  where the columns are subtraction sequences and the top two squares are cartesian, then the pushouts along the rows form a subtraction sequence
  \[
  \xymatrix{
    X' \amalg_{W'} Y' \ar[r] & X \amalg_W Y & X'' \amalg_{W''} Y'' \ar[l]
    }
  \]

  \end{enumerate}
\end{defn}

We need an appropriate notion of functor between two subtractive categories.

\begin{defn}
  A functor $F: \mc{C} \to \mc{D}$ of subtractive categories is \textbf{exact} if
  \begin{enumerate}
  \item $F$ preserves the initial object: $F(\emptyset) = \emptyset$
  \item $F$ preserves subtraction sequences: If $X \hookrightarrow Z \leftarrow Y$ is a subtraction sequence, then
    \[
    F(X) \hookrightarrow F(Z) \leftarrow F(Y)
    \]
    is a subtraction sequence.
  \item $F$ preserves cocartesian diagrams. 
  \end{enumerate}
\end{defn}

\begin{rmk}
In Waldhausen's work, item 2 is subsumed by item 3. In our case quotients (i.e. pushouts along a map to the final object) and subtraction are not the same, so we must posit an extra condition. 
\end{rmk}

The work of Section 2 gives us the following. 

\begin{cor}
  $\mbf{Sch}_X$ and $\mbf{Var}_{/X}$ are subtractive categories with cofibrations the closed inclusions. 
\end{cor}
\begin{proof}
  First, these are all categories with subtraction by Example \ref{varieties_are_categories_with_subtraction}.
  
Furthermore, pushouts diagrams where both legs are closed immersions exist Thm. \ref{pushouts_exist}, the pushout product axiom holds Prop.\ref{pushout_product}, and the final axiom regarding subtraction and pushout holds Prop. \ref{var_sub_pushout}. Thus $\mbf{Sch}_{/X}$ and $\mbf{Var}_{/X}$ are subtractive categories. 
\end{proof}

\begin{rmk}
Note that $\mbf{Sch}^{sm}_{/X}$ is \textit{not} a subtractive category, as pushing out along closed inclusions introduces singularities. 
\end{rmk}

As in Waldhausen \cite[Lem. 1.1.1] {waldhausen}, we will proceed to show that the arrow category $F_1 \mc{C}$ of a subtractive category is also a subtractive category.

\begin{defn}
Let $\mc{C}$ be a subtractive category. Let $F_1 \mc{C}$ denote the category with objects cofibrations $Z \hookrightarrow X$ and morphisms cartesian diagrams. 
\end{defn}

\begin{prop}
Let $f: (W \hookrightarrow X) \to (Y \hookrightarrow Z)$ be a map in $F_1 \mc{C}$. We define $f$ to be a cofibration if $W \to Y$ and $X \to Z$ are. This turns $F_1 \mc{C}$ into a subtractive category. 
\end{prop}
\begin{proof}
  First, the cofibrations form a category. This is clear by the usual properties of pullbacks. The category $F_1 \mc{C}$  has an initial object $(\emptyset \to \emptyset)$ and the isomorphisms are cofibrations. Pullbacks exist and are defined point-wise and are easily seen to satisfy cobase change. Furthermore, subtractions exist by the pullback axiom: given a diagram below we consider the left square to be a cofibration in $F_1 \mc{C}$ and the right vertical map will be the corresponding subtraction guaranteed by the axioms:
  \[
  \xymatrix{
    Z \ar@{^{(}->}[d] \ar@{^{(}->}[r] & X \ar@{^{(}->}[d] & X - Z \ar@{^{(}->}[d]\ar[l]\\
    Z' \ar@{^{(}->}[r] & X' & X'-Z' \ar[l]
  }
  \]

  This proves that $F_1 \mc{C}$ is a category with subtraction.

  To see that it is a subtractive category, we note that pushouts can be defined point-wise. The map produced by the pushout is a cofibration by the subtraction and pushout axiom. Pushout product follows from the definition of pullback, and the pushout product in $\mc{C}$. The interaction of subtraction and pushout follows easily, though tediously, from the work above.

\end{proof}

For future use, we introduce one more new category. 

\begin{defn}
  Let $F^+_1 \mc{C}$ denote the category whose objects are subtraction sequences $Z \hookrightarrow X \xleftarrow{\circ} Y$ in $\mc{C}$ and morphisms are diagrams
  \begin{equation}\label{morphism_f1C}
  \xymatrix{
    Z \ar@{^{(}->}[r]\ar[d]  & X \ar[d] & Y \ar[l]_\circ \ar[d] \\
    Z'\ar@{^{(}->}[r]  & X' & Y' \ar[l]_\circ
  }
  \end{equation}
  where both squares are cartesian. 
\end{defn}

\begin{defn}\label{stq}
  We define three functors $s, t, q: F^+_1 \mc{C} \to \mc{C}$ on objects
  \begin{enumerate}
  \item $s (Z \hookrightarrow X \leftarrow Y) = Z$
  \item $t (Z \hookrightarrow X \leftarrow Y) = X$
  \item $q (Z \hookrightarrow X \leftarrow Y) = Y$
  \end{enumerate}
\end{defn}

\begin{lem}
The functors $s, t, q$ are exact. 
\end{lem}
\begin{proof}
  Only the fact that $q$ is exact requires proof. First, $q$ takes cofibrations to cofibrations. To see this, note that a cofibration is a diagram such as (\ref{morphism_f1C}) where all vertical arrows are cofibrations. That $q$ takes subtraction sequences to subtraction sequences is a consequence of the preservation of subtraction sequences under pullback. That $q$ preserves cocartesian diagrams is exactly Defn. \ref{subtractive_category} Axiom 3.
\end{proof}

\subsection{$SW$-categories and $\widetilde{S}_\bullet$}

We finally introduce a modification of Waldhausen's $S_\bullet$-construction. Before doing so, we define a type of subtractive category where we allow for the presence of weak equivalence. In the main example in this paper, the category $\mbf{Var}_{/k}$, the weak equivalences will simply be the isomorphisms. We introduce the definition below in order to allow for weaker notions of equivalence (e.g. birational equivalence) in future work. 

\begin{defn}
  An \textbf{SW-category} (subtractive Waldhausen category) is a subtractive category equipped with a category of weak equivalences, $w \mc{C}$,  such that
  \begin{enumerate}
  \item The isomorphisms are contained in $w\mc{C}$
  \item Gluing holds: Given the diagram where all horizontal arrows are cofibrations
    \[
    \xymatrix{
      Y\ar[d]_{\simeq} & X\ar[d]_{\simeq} \ar@{_{(}->}[l]\ar@{^{(}->}[r] & Z\ar[d]_{\simeq}\\
      Y' & X'\ar@{_{(}->}[l]\ar@{^{(}->}[r] & Z'
      }
    \]
    we have
    \[
    Y \amalg_X Z \simeq Y' \amalg_{X'} Z'
    \]
  \item Subtraction is respected: If we have a commuting square
    \[
    \xymatrix{
      X \ar@{^{(}->}[r] \ar[d]_{\simeq} & Y \ar[d]^{\simeq}\\
      X' \ar@{^{(}->}[r] & Y'
    }
      \]
      then there is an induced weak equivalence $X - Y \xrightarrow{\simeq} X' - Y'$. 
  \end{enumerate}
\end{defn}

\begin{defn}
 A functor $F: \mc{C} \to \mc{D}$ between SW-categories is \textbf{exact} if $F$ preserves weak equivalences and $F$ is exact as a functor of subtractive categories.  
\end{defn}

\begin{rmk}
Note that for any subtractive category $\mc{C}$, if we declare the isomorphisms in $\mc{C}$ to be the weak equivalences, we obtain an SW-category.
\end{rmk}

Now, the development above proves

\begin{prop}\label{var_SW_cat}
Let $X$ be a scheme and let $\mbf{Var}_{/X}$ be the category of separated, finite-type schemes over $X$. Then $\mbf{Var}_{/X}$ is an SW-category with cofibrations the closed immersions and weak equivalences the isomorphisms of schemes. 
\end{prop}

We proceed to give the version of Waldhausen's construction of $K$-theory appropriate to SW-categories. This will be a modification of his $S_\bullet$-construction.

To cleanly state the construction we need to define a useful indexing category.

\begin{defn}
  Let $[n]$ denote the ordered set $\{0, \dots, n\}$ considered as a category, i.e. there is a map $i \to j$ if $i \leq j$. Define $\widetilde{\operatorname{Ar}}[n]$ to be the full subcategory of $[n]^{\text{op}} \times [n]$ consisting of pairs $(i, j)$ with $i \leq j$. 
\end{defn}

\begin{example}
  $\widetilde{\operatorname{Ar}}[2]$ may be visualized as
  \[
  \xymatrix{
    (0,0) \ar[r] & (0,1)\ar[r] & (0,2)\\
    & (1,1)\ar[u] \ar[r] & (1,2)\ar[u]\\
    &              & (2,2) \ar[u]
    }
  \]
  and will be referred to colloquially as ``flags'' below. 
\end{example}

\begin{defn}[$\widetilde{S}_\bullet$-construction]\label{s_dot}
   Let $\mc{C}$ be an SW-category. We define $\widetilde{S}_n \mc{C}$ to be the set of functors
  \[
  X: \widetilde{\operatorname{Ar}}[n] \to \mc{C}
  \]
  subject to the conditions
  \begin{itemize}
  \item $X_{i,i} = \emptyset$, the empty variety.
  \item Every $X_{i,j} \to X_{i,k}$ where $j < k$ is a cofibration.
  \item  The sub-diagram
  \[
  X_{i,j} \to X_{i,k} \leftarrow X_{j, k}
  \]
  is a subtraction sequence.
\item  For $i < j < k < l$, the subdiagram
  \[
  \xymatrix{
    X_{ik} \ar@{^{(}->}[r] & X_{il}\\
    X_{jk} \ar@{^{(}->}[r]\ar[u]^{\circ} & X_{jl} \ar[u]^{\circ}
  }
  \]
  is cartesian. 
  \end{itemize}
  This defines a simplicial set as follows. The face maps are
  \begin{enumerate}
  \item $d_0: \widetilde{S}_n \mc{C} \to \widetilde{S}_{n-1} \mc{C}$ is given by removing the first row. 
  \item $d_k: \widetilde{S}_n \mc{C} \to \widetilde{S}_{n-1} \mc{C}$ is given by deleting the $k$th row and column and composing the remaining maps. 
  \end{enumerate}
  The $i$th degeneracy maps are given by inserting identity maps $X_{i,j} \xrightarrow{=} X_{i, j}$ for all $j$. From this it is clear that the simplicial relations hold.
\end{defn}

In fact, $\widetilde{S}_\bullet \mc{C}$ can be considered as a simplicial category (i.e. a simplicial object in categories). First, we introduce some notation. Let $i_0: [n] \hookrightarrow \widetilde{\operatorname{Ar}}[n]$ be given by $j \mapsto (0, j)$. 

\begin{defn}
We consider $\widetilde{S}_n \mc{C}$ as a category as follows. The objects are the functors $X: \widetilde{\operatorname{Ar}}[n] \to \mc{C}$ as above and the morphisms are functors $Y: \widetilde{\operatorname{Ar}}[n] \times [1] \to \mc{C}$ with the additional restriction that all squares in $i^\ast_0 Y : [n] \times [1]\to \mc{C}$ are cartesian. Composition is given in the obvious way. 
\end{defn}

\begin{rmk}
The requirement that the squares $[n]\times[1] \to \mc{C}$ be cartesian is a consequence of the fact that we only have functoriality of subtraction with respect to cartesian squares. 
\end{rmk}

\begin{rmk}
This makes $\widetilde{S}_\bullet \mc{C}$ into a simplicial \textit{category}. 
\end{rmk}

The category $\widetilde{S}_n \mc{C}$, built from a subtractive category, can itself be given the structure of a subtractive category. 

\begin{lem}
The category $\widetilde{S}_n \mc{C}$ is a subtractive category. The category of cofibrations is given by functors $Y: \widetilde{\operatorname{Ar}}[n] \times [1] \to \mc{C}$ such that the restriction $Y((i,j)) : [1] \to \mc{C}$ are cofibrations. The category of fibrations are given similarly. The subtraction sequences are level-wise subtraction sequences. 
\end{lem}
\begin{proof}
This is entirely analgous to the proof for $F_1 \mc{C}$. 
\end{proof}

\begin{rmk}
This will allow us to iterate the $S_\bullet$-construction. 
\end{rmk}

We can finally define our \textit{space} $K(\mc{C})$ --- more machinery is needed to prove that it may be delooped.

\begin{defn}
Let $\mc{C}$ be an SW-category. We define the space $\underline{K}(\mc{C})$ to be $\Omega |w_\bullet \widetilde{S}_\bullet \mc{C}|$, where $|-|$ denote the simplicial realization of a bisimplicial set. Here, $w \mc{C}$ denotes the subcategory of all objects with maps weak equivalences, and $w_\bullet \mc{C}$ denotes the simplicial nerve of that category. 
\end{defn}

Of course, the salient property of this space holds. 

\begin{prop}
$\pi_0 \underline{K}(\mc{C}) = K_0 (\mc{C})$. 
\end{prop}
\begin{proof}
  This follows by standard methods; see, for example, \cite{weibel}. For any simplicial space (or bisimplicial set) $X_\bullet$, we can compute $\pi_1 |X_\bullet|$ via generators and relations:
  \[
  \pi_1 |X_\bullet| =  \langle \pi_0 (X_1) \rangle / (d_1 (x) = d_2 (x) + d_0 (x)) \qquad x \in \pi_0 (X_2). 
  \]
  Here our simplicial space is $X_n = |i_\bullet \widetilde{S}_n \mc{C}|$. Therefore, $\pi_0 (X_1)$ is the set of equivalence classes of varieties up to isomorphism. Also, $X_2$ is the set of equivalences classes of subtraction sequences. For a subtraction $X \hookrightarrow Y \leftarrow X - Y$, call it $c$, $d_0 (c) = Y- X$, $d_1 (c) = Y$ and $d_2 (c) = X$. Therefore, the relations are
  \[
  [Y] = [X] + [Y-X]. 
  \]
\end{proof}

\begin{rmk}
We can define $\underline{K}(\mc{C})$ for $\mc{C}$ any category with subtraction. We have not yet used any other structure. However, in order that the space $\underline{K}(\mc{C})$ deloop to a spectrum $K(\mc{C})$, we will need $\mc{C}$ to be a subtractive or SW-category. 
\end{rmk}

We now produce $K(\mc{C})$ as a symmetric spectrum by iterating the $\widetilde{S}_\bullet$ construction in an appropriate way; that is, the following is what one gets if we consider $\widetilde{S}_\bullet \mc{C}$ as an SW-category and iterate the $\widetilde{S}_\bullet$-construction. We will show in the subsequent section that this is a quasi-fibrant symmetric spectrum \cite{hovey_shipley_smith, mmss}.

\begin{defn} \label{iterated_S_dot_construction}
  Let $\mc{C}$ be an SW-category. We consider the category of functors
  \[
  F : \widetilde{\text{Ar}}[n_1] \times \cdots \times \widetilde{\text{Ar}}[n_k] \to \mc{C}.
  \]
  and write each object of $\widetilde{\text{Ar}}[n_\ell]$ as $(i_\ell, j_\ell)$. Let $S^k_{n_1, \dots, n_k} \mc{C}$ be the full subcategory consisting of functors $F$ such that
  \begin{enumerate}
  \item $F ((i_1, j_1), \dots, (i_k, j_k)) = \ast$ whenever $i_\ell = j_\ell$ for some $\ell$.
  \item The subfunctor $F((0, i_1), \dots, (0,i_k)): [n_1] \times \cdots \times [n_k] \to \mc{C}$ defines a cube such that every sub-face is cartesian. 
  \item Given $((i_1, j_1), \dots, (i_k, j_k))$ in $ \widetilde{\text{Ar}}[n_1] \times \cdots \times \widetilde{\text{Ar}}[n_k]$ and $1 \leq \ell \leq k$ and every $j_\ell \leq m \leq n_\ell$ the sequence
    \[
    \xymatrix{
      F((i_1, j_1), \dots, (i_\ell, j_\ell), \dots, (i_k, j_k)) \ar[r] &  F((i_1, j_1), \dots, (i_\ell, m), \dots, (i_k, j_k))\\
       & F((i_1, j_1), \dots, (j_\ell, m), \dots, (i_k, j_k)) \ar[u] 
      }
    \]
    is a subtraction sequence. 
  \item Given $i_l < n < j_l < m$ the diagram
    \[
    \xymatrix{
    F((i_1, j_1), \dots, (i_l, j_l), \dots, (i_k, j_k)) \ar[r] & F((i_1, j_1), \dots, (i_l, m), \dots, (i_k, j_k))\\
    F((i_1, j_1), \dots, (n, j_l), \dots, (i_k, j_k)) \ar[u]\ar[r] & F((i_1, j_1), \dots, (n,m), \dots, (i_k, j_k))\ar[u]
    }
    \]
    is cartesian.
  \end{enumerate}

\end{defn}

Using this we may define the symmetric spectrum $K(\mc{C})$

\begin{defn}\label{K_symmetric_spectrum}
  Let $\mc{C}$ be an SW-category and define
  \[
  K(\mc{C})(k) = |N_\bullet (w S^{(k)}_{\bullet, \dots, \bullet} \mc{C})|. 
  \]
  This space has a $\Sigma_k$-action given by permuting the simplicial directions. 
\end{defn}

\section{Additivity}

The slogan for algebraic $K$-theory is that it is the universal machine to split exact sequences. A more precise statement is that the ``Additivty Theorem'' holds and that $K$-theory is the universal functor for which this theorem holds.  This was recently proven in \cite{blumberg_gepner_tabuada,barwick} though it has been a guiding principle of the field since its inception. As one would expect, almost every other standard property of $K$-theory follows from additivity \cite{staffeldt}. In our situation, we cannot hope to prove the array of theorems that additivity usually provides; we settle for using it to prove that $K(\mbf{Var}_{/k})$ is in infinite loop space.

The additivity theorem for SW-categories is as follows. This section will be devoted to the proof of this theorem.

\begin{thm}[Additivity]\label{additivity}
Let $\mc{C}$ be an SW-category. Consider the map
\[
A = (s, q): F^+_1(\mc{C}) \to \mc{C} \times \mc{C}. 
\]
Upon applying $\widetilde{S}_\bullet$ we get a homotopy equivalence of simplicial sets
\[
\widetilde{S}_\bullet F^+_1(\mc{C}) \xrightarrow{\sim} \widetilde{S}_\bullet \mc{C} \times \widetilde{S}_\bullet \mc{C}. 
\]
\end{thm}

For Waldhausen categories, the cleanest proof of additivity is due to McCarthy \cite{mccarthy}. We will mimic his proof to show that additivity holds for SW-categories; the key point is that while pushouts are used extensively in the proof, \textit{only} pushouts where \textit{both} legs are cofibrations are needed. As pointed out in Section 2, these are exactly the types of pushouts that we do have.

We pause here to recall the definition of a simplicial homotopy, since it will be used frequently below.

\begin{defn}\label{simplicial_homotopy}
  Let $X, Y$ be simplicial sets and $f, g: X \to Y$ simplicial maps. A \textbf{simplicial homotopy} is a simplicial map $X \times \Delta^1 \to Y$ such that restricting to the first vertex of $\Delta^1$ gives $f$ and restricting to the second vertex gives $g$.

  These requirements can be packaged combinatorially as follows. A simplicial homotopy is a family of maps $h_i: X_n \to Y_{n+1}$ with $0 \leq i \leq n$ for each $n$. The following identities are required to hold:
  \[
\begin{cases}
  d_0 h_0 = f \\
  d_{n+1} h_n = g
\end{cases}
\ \
\begin{cases}
  d_i h_j = h_{j-1} d_i & i < j\\
  d_{j+1} h_{j+1} = d_{j+1} h_j &  \\
  d_i h_j = h_j d_{i-1} & i > j+1
\end{cases}
\ \
\begin{cases}
  s_i h_j = h_{j+1} s_i & i \leq j \\
  s_i h_j = h_j s_{i-1} & i > j 
\end{cases}
\]

\end{defn}

We begin with a useful construction.

\begin{defn}\label{mixing_category}
  Let $\mc{C}, \mc{D}$ be categories with subtraction and let $f: \mc{C} \to \mc{D}$ be an exact functor. Define a bisimplicial set $\mc{C} \otimes_{S_\bullet f} \mc{D}$ by setting \[(\mc{C} \otimes_{S_\bullet f} \mc{D})([m],[n])\] to be pairs of diagrams in $S_m \mc{C}$ and $S_{m+n} \mc{D}$ (we are omitting the rows below the first):
  \begin{equation}\label{mixing_diagram}
  \xymatrix{
    X_0 \ar@{^{(}->}[r] & X_1\ar@{^{(}->}[r] & \cdots\ar@{^{(}->}[r] & X_m & &  \\
    Y_0 \ar@{^{(}->}[r] & Y_1\ar@{^{(}->}[r] & \cdots\ar@{^{(}->}[r] & Y_m \ar@{^{(}->}[r]& \cdots\ar@{^{(}->}[r] & Y_{m+n} 
  }
  \end{equation}
  such that $f(X_i) = Y_i$ and $f(X_i \to X_{i+i}) = Y_i \to Y_{i+1}$. The face and degeneracy maps are given by composition and repetition, respectively. 
\end{defn}

\begin{defn}
  Let $X_\bullet$ be a simplicial set. Then $X^R$ will denote a bisimplicial set $X^R ([m],[n]) = X([n])$. Similarly, $X^L$ will denote the bisimplicial set $X^L([m],[n]) = X([m])$. 
\end{defn}

\begin{defn}
  We define a bisimplicial map $\rho: \mc{C} \otimes_{\widetilde{S}_\bullet f} \mc{D} \to \widetilde{S}_\bullet \mc{D}^R$ by taking (\ref{mixing_diagram}) to
  \[
  Y_{m+1} - Y_{m+1} \hookrightarrow Y_{m+2}-Y_{m+1} \hookrightarrow \cdots \hookrightarrow Y_{m+n}-Y_{m+1}
  \]
\end{defn}

\begin{prop}\cite[p.326]{mccarthy}
The following are equivalent
\begin{enumerate}
\item $\widetilde{S}_\bullet f: \widetilde{S}_\bullet \mc{C} \to \widetilde{S}_\bullet \mc{D}$ is a homotopy equivalence
\item The map $\rho : \mc{C} \otimes_{\widetilde{S}_\bullet f} \mc{D} \to \widetilde{S}_\bullet \mc{D}^R$ is a homotopy equivalence. 
\end{enumerate}
\end{prop}
\begin{proof}

  Consider the commutative diagram of bisimplicial sets:
  \[
  \xymatrix{
    \widetilde{S}_\bullet \mc{D}^R \ar@{=}[d] & \mc{C} \otimes_{\widetilde{S}_\bullet f} \mc{D}\ar[l] \ar[d]^f \ar[r]^{\mbf{1}} & \widetilde{S} _\bullet \mc{C}^L\ar[d]^f \\
    \widetilde{S}_\bullet \mc{D}^R & \mc{D} \otimes_{\widetilde{S}_\bullet \id} \mc{D} \ar[r]^{\mbf{2}}\ar[l]_{\mbf{3}} & \widetilde{S}_\bullet \mc{D}^L
  }
  \]
  The map labelled $\mbf{1}$ is obtained by forgetting the ``$\mc{D}$''-portion of $\mc{C}\otimes_{S_\bullet f} \mc{D}$. We now fix $m$ to obtain maps between simplicial sets (indexed by $n$)
  \[
  (\mc{C} \otimes_{\widetilde{S}_\bullet f} \mc{D})([m],[n]) \to \widetilde{S}_\bullet \mc{C}^L ([m],[n]) = \widetilde{S}_\bullet \mc{C} ([m]). 
  \]
  The simplicial set on the right is constant. The simplicial set on the left is homotopy equivalent to $\widetilde{S}_m \mc{C}$. To see this, fix the $\widetilde{S}_m \mc{C}$ portion of the pair and consider the resulting simplicial set. It is contractible by the same argument that contracts the nerve of a category with an initial object. Thus, levelwise, $\mc{C}\otimes_{\widetilde{S}_\bullet f} \mc{D}$ and $S_\bullet \mc{C}^L$ are equivalent, and thus homotopy equivalent as bisimplicial sets by the realization lemma. The maps $\mbf{1}$ and $\mbf{2}$ are shown to be homotopy equivalences in exactly the same way.

  Thus, the vertical right arrow will be a homotopy equivalence if and only if the upper left horizontal arrow is a homotopy equivalence. 

\end{proof}

This reduces the study of homotopy equivalences $\widetilde{S}_\bullet \mc{C} \to \widetilde{S}_{\bullet} \mc{D}$ to the study of the maps $\rho$. 

Now define a self-map 
\[
E_n: (\mc{C} \otimes_{S_\bullet F} \mc{D})(-,[n]) \to (\mc{C} \otimes_{S_\bullet F} \mc{D})(-,[n])
\]
via a subtracting procedure. We take the standard diagrams (\ref{mixing_diagram}) to  (again, omitting the rows below the first)
\[
\xymatrix@R=.3cm{
\emptyset \ar@{=}[r] & \cdots\ar@{=}[r] & \emptyset &                          &                          &           \\
\emptyset \ar@{=}[r] & \cdots \ar@{=}[r] & \emptyset \ar@{=}[r]& Y_{m+1} - Y_{m+1} \ar@{^{(}->}[r] & Y_{m+2} - Y_{m+1} \ar@{^{(}->}[r] & \cdots\ar@{^{(}->}[r] & Y_{m+n} - Y_{m+1} }
\]

The above proposition implies

\begin{cor}\cite[p.326]{mccarthy}
If $E_n$ are homotopy equivalences for all $n$, then $\widetilde{S}_\bullet F: \widetilde{S}_\bullet \mc{C} \to \widetilde{S}_\bullet \mc{D}$ is a homotopy equivalence. 
\end{cor}
\begin{proof}

  For a fixed $n$ define a map $I_n: \widetilde{S}_\bullet \mc{D}^L ([m],[n]) \to \mc{C} \otimes_{\widetilde{S}_\bullet f} \mc{D} ([m],[n])$ by sending
  \[
  \emptyset = Y_0 \hookrightarrow Y_1 \hookrightarrow \cdots \hookrightarrow Y_n
  \]
  to
  \[
  \xymatrix@R=.3cm{
    Y_0 \ar@{=}[r] & \cdots & Y_0  \\
    Y_0 \ar@{=}[r] & \cdots & Y_0 \ar@{=}[r] & Y_0 \ar@{^{(}->}[r] & Y_2 \ar@{^{(}->}[r] & \cdots\ar@{^{(}->}[r] & Y_n 
  }
  \]
  Note that $\rho \circ I_n  = \id$ and $I_n \circ \rho = E_n$. If $E_n$ are homotopy equivalences, then so are $\rho$ and $I_n$. But if $\rho$ is a homotopy equivalence then $\widetilde{S}_\bullet f$ is as well. 
\end{proof}

Let $A: F^+_1(\mc{C}) \xrightarrow{(s,t)} \mc{C} \times \mc{C}$ be the functor defined by the additivity functors (Defn. \ref{stq}). In order to use the techniques above to work with this map, we need to consider the category
\[
F^+_1(\mc{C}) \otimes_{S_\bullet A} \mc{C}^2. 
\]
and diagrams in this category. Here is their typical form (as always omitting the rows after the first)
\begin{equation}\label{additivity_diagram}
\xymatrix@C=.5cm@R=.5cm{
  \emptyset\ar@{=}[r] & A_0 \ar@{^{(}->}[r]\ar@{^{(}->}[d] & A_1 \ar@{^{(}->}[r]\ar@{^{(}->}[d] & \cdots \ar@{^{(}->}[r] & A_m\ar@{^{(}->}[d]\\
  \emptyset \ar@{=}[r]& C_0 \ar@{^{(}->}[r] & C_1\ar@{^{(}->}[r] & \cdots\ar@{^{(}->}[r] & C_m\\
 \emptyset \ar@{=}[r]&B_0\ar@{^{(}->}[r]\ar[u]_{\circ} & B_1\ar@{^{(}->}[r] \ar[u]_{\circ} & \cdots\ar@{^{(}->}[r] & B_m \ar[u]_{\circ}\\
  \emptyset \ar@{=}[r]& A_0\ar@{^{(}->}[r] & A_1\ar@{^{(}->}[r] & \cdots\ar@{^{(}->}[r] & A_m\ar@{^{(}->}[r] & S_0\ar@{^{(}->}[r] & \cdots\ar@{^{(}->}[r] & S_n\\
  \emptyset\ar@{=}[r] & B_0\ar@{^{(}->}[r] & B_1\ar@{^{(}->}[r] & \cdots\ar@{^{(}->}[r] & B_m\ar@{^{(}->}[r] & T_0 \ar@{^{(}->}[r] & \cdots \ar@{^{(}->}[r] & T_n
}
\end{equation}

\begin{rmk}\label{diagram_labelling}
In what follows, we will need to refer to the rows below the pictured rows in (\ref{additivity_diagram}) --- the pictured rows are the zeroth rows of a flag. The elements in the $k$th row will be referred to by $A_{k,l}$, $B_{k,l}$ and $C_{k,l}$. 
\end{rmk}

\begin{rmk}
We note that the only difference between these diagrams and the diagrams that appear in \cite{mccarthy} is the fact that the arrows between the $B$s and $C$s go in opposite directions. 
\end{rmk}

We will now show that $E_n$ for the functor $(s,q): F^+_1(\mc{C}) \to \mc{C} \times \mc{C}$ is a homotopy equivalence. Recall that $E_n$ will be a map
\[
E_n : F^+_1 (\mc{C}) \otimes_{S_\bullet A}  \mc{C}^2 \to  F^+_1 (\mc{C}) \otimes_{S_\bullet A}  \mc{C}^2
\]

To show that this is a weak equivalence, McCarthy defines a map of simplicial sets
\[
\Gamma: F^+_1 (\mc{C}) \otimes_{S_\bullet A} \mc{C}^2(-,[n]) \to F^+_1 (\mc{C}) \otimes_{S_\bullet A} \mc{C}^2 (-,[n])
\]
and shows
\begin{enumerate}
\item $\Gamma$ is a retraction onto some subspace $\mc{X} \subset F^+_1 (\mc{C}) \otimes_{\widetilde{S}_\bullet A} \mc{C}^2 (-,[n])$
\item $\Gamma \simeq \id$
\item $E_n|_{\mc{X}} \simeq \id_{\mc{X}}$
\item $E_n \circ \Gamma = E_n$
\end{enumerate}

Taken together, these implies that $E_n$ is a homotopy equivalence. 

We will use exactly the same procedure here. 

\begin{defn}
  The map $\Gamma: F^+_1(\mc{C}) \times_{S_\bullet A} \mc{C}^2 (-,[n]) \to F^+_1(\mc{C}) \times_{S_\bullet A} \mc{C}^2 (-,[n])$ is defined by taking diagrams (\ref{additivity_diagram}) to diagrams (as always, omitting rows below the first) to
  \begin{equation}\label{En_diagram}
  \xymatrix@C=.5cm@R=.5cm{
    \emptyset \ar@{=}[r]& \emptyset \ar@{=}[r]\ar@{^{(}->}[d] & \emptyset \ar@{=}[r]\ar@{^{(}->}[d] & \cdots\ar@{=}[r] & \emptyset\ar@{^{(}->}[d]\\
    \emptyset \ar@{=}[r] & B_0 \ar@{=}[d] \ar@{^{(}->}[r]   & B_1\ar@{=}[d]\ar@{^{(}->}[r] & \cdots\ar@{^{(}->}[r] & B_m\ar@{=}[d]\\
    \emptyset \ar@{=}[r] & B_0 \ar@{^{(}->}[r] & B_1 \ar@{^{(}->}[r]& \cdots\ar@{^{(}->}[r] & B_m \\
    \emptyset\ar@{=}[r] & \emptyset \ar@{=}[r]& \emptyset\ar@{=}[r] & \cdots\ar@{=}[r] & \emptyset\ar@{=}[r] & S_0 - S_0\ar@{^{(}->}[r] & S_1 - S_0 \ar@{^{(}->}[r]& \cdots\ar@{^{(}->}[r] & S_m - S_0 \\
    \emptyset\ar@{=}[r] & B_0\ar@{^{(}->}[r] & B_1\ar@{^{(}->}[r] & \cdots\ar@{^{(}->}[r] & B_m \ar@{^{(}->}[r]& T_0 \ar@{^{(}->}[r]& T_1 \ar@{^{(}->}[r]& \cdots\ar@{^{(}->}[r] & T_n
  }
  \end{equation}
\end{defn}

We have already defined $E_n$ in general above, but it is useful to spell out what it is in this context.

\begin{defn}
  $E_n$ takes diagrams of the form (\ref{additivity_diagram}) to
  \[
  \xymatrix@C=.3cm@R=.3cm{
    \emptyset \ar@{=}[r]\ar@{^{(}->}[d] & \cdots \ar@{=}[r] & \emptyset\ar@{^{(}->}[d]\\
    \emptyset \ar@{=}[r] & \cdots \ar@{=}[r] & \emptyset\\
    \emptyset \ar@{=}[r]\ar[u] & \cdots \ar@{=}[r] & \emptyset \ar[u]\\
    \emptyset \ar@{=}[r] & \cdots\ar@{=}[r] & \emptyset \ar@{=}[r] & S_0 - S_0\ar@{^{(}->}[r] & S_1 - S_0\ar@{^{(}->}[r] & \cdots\ar@{^{(}->}[r] & S_n - S_0 \\
    \emptyset \ar@{=}[r] & \cdots\ar@{=}[r] &\emptyset \ar@{=}[r] & T_0 - T_0\ar@{^{(}->}[r] & T_1 - T_0 \ar@{^{(}->}[r]& \cdots\ar@{^{(}->}[r] & T_n - T_0
  }
  \]
\end{defn}

\begin{defn}
Let the subspace $\mc{X} \subset F^+_1(\mc{C}) \times_{S_\bullet A} \mc{C}^2 (-,[n])$ denote the subspace where all of the $A_i$ are $\emptyset$. 
\end{defn}

\begin{prop}
$E_n|_{\mc{X}} \simeq \id_{\mc{X}}$
\end{prop}
\begin{proof}
This is done by exactly the same argument that contracts a category with final object. One contracts the string of $B$s in (\ref{En_diagram}) to $T_0$. 
\end{proof}

 We note that $\Gamma$ is a retraction of $F^+_1(\mc{C}) \times_{S_\bullet A} \mc{C}^2 (-,[n])$ onto $\mc{X}$. To complete the proof that $E_n$ is a homotopy equivalence, and thus additivity, we  need to show that $\Gamma$ is homotopic to the identity.

This is done by producing an explicit simplicial homotopy
\[
h: (E(\mc{C}) \times_{S_\bullet A} \mc{C}^2(-,[n])) \times \Delta^1 \to (E(\mc{C}) \times_{S_\bullet A} \mc{C}^2 (-,[n]))
\]
Recall that a simplicial homotopy can be expressed in a combinatorial fashion (Defn. \ref{simplicial_homotopy}) via maps $h_i$. We fix $m$ and for  $e \in E(\mc{C}) \otimes_{S_\bullet A} \mc{C}^2 ([m],[n])$ (which recall is of the form (\ref{additivity_diagram})) and we define $h_i (e)$ with $0 \leq i \leq m$ to be
\begin{equation}\label{main_homotopy_diagram}
\xymatrix@C=.2cm @R=.5cm{
0=A_0 \ar@{^{(}->}[r] \ar@{^{(}->}[d] & A_1 \ar@{^{(}->}[r] \ar@{^{(}->}[d] & \cdots & A_i \ar@{^{(}->}[d] \ar@{^{(}->}[r] & S_0 \ar@{=}[r] \ar@{^{(}->}[d] & S_0 \ar@{=}[r] \ar@{^{(}->}[d] & \cdots \ar@{=}[r] & S_0 \ar@{^{(}->}[d]\\
C_0 \ar@{^{(}->}[r] & C_1\ar@{^{(}->}[r] & \cdots & C_i \ar@{^{(}->}[r] & C_i \amalg_{A_i} S_0 \ar@{^{(}->}[r] & C_{i+1} \amalg_{A_{i+1}} S_0  & \cdots \ar@{^{(}->}[r] & C_m \amalg_{A_m} S_0\\
B_0 \ar[u]^{\circ} \ar@{^{(}->}[r] & B_1 \ar[u]^{\circ} \ar@{^{(}->}[r] & \cdots & B_i \ar[u]^{\circ} \ar@{=}[r] & B_i \ar[u]^{\circ} \ar@{^{(}->}[r] & B_{i+1} \ar@{^{(}->}[r] \ar[u]^{\circ} & \cdots \ar@{^{(}->}[r] & B_m \ar[u]^{\circ}\\
 0=  A_0 \ar@{^{(}->}[r] & A_1\ar@{^{(}->}[r] & \cdots  \ar@{^{(}->}[r] &  A_i \ar@{^{(}->}[r] & S_0 \ar@{=}[r] & S_0 \ar@{=}[r] & \cdots \ar@{=}[r] & S_0\\
0 = B_0 \ar@{^{(}->}[r] & B_1 \ar@{^{(}->}[r] & \cdots \ar@{^{(}->}[r] & B_i \ar@{=}[r] & B_i \ar@{^{(}->}[r] & B_{i+1} \ar@{^{(}->}[r] & \cdots \ar@{^{(}->}[r] & B_m 
}
\end{equation}

Note that here we are using the existence of pushouts provided by Th. \ref{pushouts_exist}. This is one of the critical points where that fact is used.

Although we are not displaying the levels below the upper row the the diagrams above, we will need to reference the rows below. For the diagram $e$, we retain the conventions of Rmk. \ref{diagram_labelling}: the choices of subtraction in the diagram $e$ will be refered to by $A_{k,l}$, $B_{k,l}$ and $C_{k,l}$. For the diagram $h_i (e)$ we make the convention that the symbol $h_i (e)^A$ represents the flag corresponding to the first row in (\ref{main_homotopy_diagram}), and similarly for $h_i (e)^B$ and $h_i (e)^C$. Thus, the hidden parts of the flags are indexed by $h_i(e)^A_{k,l}, h_i(e)^B_{k,l}, h_i(e)^C_{k,l}$.  We now explicitly identify these flags. For $i \geq 0$ define 

\begin{align*}
  h_i (e)^A_{k,l} &=
  \begin{cases}
    A_{k,l} & k, l \leq i \\
    S_0 - A_{0,k} & k \leq i, l > i\\
    \emptyset & \text{otherwise}
    \end{cases} \\
  h_i (e)^C_{k,l} &=
  \begin{cases}
    C_{k,l} & k, l \leq i \\
    C_{k,l-1} \amalg_{A_{k,l-1}} h_i (e)^A_{k,l} & k < i, l > i\\
    h_i(e)_{k,l} & k= i, l = i+1\\
    C_{k,l-1} \amalg_{A_{k,l-1}} h_i(e)^A_{k,l} & k = i, l > i+1\\
    h_i(e)^B_{k,l} & l, k > i
    \end{cases} \\
  h_i (e)^B_{k,l} &=
  \begin{cases}
    B_{k,l} & k, l \leq i \\
    B_{k,i} & l = i+1, k = i+1, l \neq k \\
    \emptyset & l = k = i+1\\
    B_{k,l-1} & \text{otherwise}
  \end{cases}
\end{align*}

The appendix depicts a few of these diagrams. 

For the most part, the maps in (\ref{main_homotopy_diagram}) are clear. One that requires comment is the map in $h_i(e)$ from $B_{k,l}$ to $C_{k,l} \amalg_{A_{k,l}} (S_0 - A_k)$. This will be the composition
\[
B_{k,l} \xrightarrow{\cong} C_{k,l} - A_{k,l} \xrightarrow{\cong} C_{k,l} \amalg_{A_{k,l}} (S_0 - A_k) - (S_0 - A_k) \hookrightarrow C_{k,l} \amalg_{A_{k,l}} (S_0 - A_k)
\]
Each of these isomorphisms and inclusions is uniquely determined by data in $e$. The other maps that require comment are those from $h^B_{k,l}$ to $h^C_{k,l}$ --- whenever both of them are $B_{k,l}$ the map between them will be the identity.

We now have to verify two assertions. The first is that the flags in (\ref{main_homotopy_diagram}) satisfy the requirements of Defn. \ref{s_dot} and the second is that $h_i$ satisfies the relations of simplicial homotopy in Defn. \ref{simplicial_homotopy}.

For the first assertion, it is clear that the flags below the rows $h_i(e)^A$ and $h_i (e)^B$ remain of the form required by Defn. \ref{s_dot}. The following proposition verifies the statement for the $h_i(e)^C$ row. 

\begin{prop}
  For any $k, l, s$ with $k < l < s$ 
  \[
  h^C_i(e)_{k,l} \to h^C_i(e)_{k,s} \leftarrow h^C_i(e)_{l,s}
  \]
  is a subtraction sequence. 
\end{prop}

\begin{proof}
  We show this in the case $k = 0$. The other cases are dealt with similarly. We proceed by dividing this into the sub-cases $l, s \leq i$, $l \leq i, s > i$ and $l, s > i$. 
  
  For $k < l < s \leq i$, the statement follows since $C_{0,l} \to C_{0,s} \leftarrow C_{l,s}$ is a subtraction sequence. 

  For $l \leq  i$, $i<s$ this is the statement that
  \[
  C_{0,l} \to C_{0,s-1} \amalg_{A_{0,s-1}} (S_0 - A_{0,0}) \leftarrow C_{0,s-1} \amalg_{A_{l,s-1}} (S_0 - A_{0,l})
  \]
  is a subtraction sequence. To see this, consider the diagram
  \[
  \xymatrix{
    C_{k,l}\ar@{^{(}->}[d] & A_{k,l} \ar@{^{(}->}[d]\ar@{^{(}->}[r]\ar@{_{(}->}[l] & A_{k,l}\ar@{^{(}->}[d]\\
    C_{k,s-1} & A_{k,s-1}\ar@{^{(}->}[r]\ar@{_{(}->}[l] & S_0 - A_{0,k}\\
    C_{l,s-1} \ar[u]^\circ & A_{l,s-1} \ar@{^{(}->}[r]\ar@{_{(}->}[l]\ar[u]^\circ & S_0 - A_{0,l} \ar[u]^\circ
  }
  \]
  The top squares are cartesian (by definition of the $A$ and $C$ flags). Thus, this satisfies Axiom 3 of Defn. \ref{subtractive_category}, and the statement follows.

  For $i < l, s$ the statement is that
  \[
  C_{0,l-1} \amalg_{A_{0,l-1}} S_0 \hookrightarrow C_{0,s-1} \amalg_{A_{0,s-1}} S_0 \leftarrow B_{l-1,s-1}
  \]
  is a subtraction sequence. To see this, we consider the diagram induced by functoriality
  \[
  \xymatrix{
    S_0 \ar@{=}[r] \ar@{^{(}->}[d] & S_0 \ar@{^{(}->}[d] & \emptyset \ar[l]\ar@{^{(}->}[d]\\
    C_{0,l-1} \amalg_{A_{0,l-1}} S_0 \ar@{^{(}->}[r] & C_{0,s-1} \amalg_{A_{0,s-1}} S_0 & B_{l-1,s-1} \ar[l]^{\circ}\\
    B_{0,l-1} \ar[u]^\circ \ar@{^{(}->}[r] & B_{0,s-1} \ar[u]^\circ & B_{l-1,s-1} \ar@{=}[u]\ar[l]_{\circ}
  }
  \]
The first and second columns are easily seen to be subtraction sequences and the top and bottom rows as well. This forces the middle row to be a subtraction sequence.

\end{proof}

We now verify that $h_i$ is a simplicial homotopy. Recall that this means that the following identities hold:

\[
\begin{cases}
  d_0 h_0 = \Gamma \\
  d_{n+1} h_n = \id 
\end{cases}
\ \
\begin{cases}
  d_i h_j = h_{j-1} d_i & i < j\\
  d_{j+1} h_{j+1} = d_{j+1} h_j &  \\
  d_i h_j = h_j d_{i-1} & i > j+1
\end{cases}
\ \
\begin{cases}
  s_i h_j = h_{j+1} s_i & i \leq j \\
  s_i h_j = h_j s_{i-1} & i > j 
\end{cases}
\]

First, it is clear that $d_{q+1} h_q = \id$. It is also clear that $d_0 h_0 = \Gamma$. 

The identities involving degeneracy hold trivially.

The middle group of identities is not hard: 

$\boxed{d_i h_j = h_{j-1} d_i}$ when $i < j$. This part only involves the $C_{k,l}$ and thus holds by the simplicial identities in the $C_{k,l}$ part of $e$. 

$\boxed{d_{j+1} h_{j+1} = d_{j+1} h_j}$. This identity is clear from the definitions.

$\boxed{d_i h_j = h_j d_{i-1}}$ when $i > j$. Again, this is not difficult. The identity comes from the simplicial identities on the $B_{k,l}$ part of $e$ and the fact that pushouts are chosen functorially and based on maps in $e$. (See the appendix for a picture).

With these verifications we know that $h_i$ is a simplicial homotopy and this ends the proof of the additivity theorem.

\subsection{Delooping}

Of course, additivity is a stepping stone to delooping for us. From the $\widetilde{S}_\bullet$ construction, we can produce a map $K(\mc{C})(k) \to \Omega K(\mc{C})(k+1)$ (the construction is reviewed below). A consequence of additivity will allow us to show that this map is a weak equivalence, which exhibits $K(\mc{C})(1)$ as an infinite loop space, and $K(\mc{C})$ as a quasi-fibrant symmetric spectrum. 

We will approach delooping as Waldhausen does. However, we need a definition first.

\begin{defn}\cite[Defn. 1.5.4]{waldhausen}
  Let $P X_\bullet$ denote the simplicial path space of the simplicial set $X_\bullet$. Then for SW-categories $\mc{A}$ and $\mc{B}$ with an exact functor $f: \mc{A} \to \mc{B}$ we define $\widetilde{S}_n (f: \mc{A} \to \mc{B})$ via pullback: 
  \[
  \xymatrix{
    \widetilde{S}_n (f: \mc{A} \to \mc{B}) \ar[r]\ar[d] & (P\widetilde{S}_\bullet \mc{B})_{n+1} \ar[d]\\
    \widetilde{S}_n \mc{A} \ar[r] & \widetilde{S}_n \mc{B}
  }
  \]
  
\end{defn}

For a simplicial path space $PX_\bullet$ there is a sequence of maps $X_1 \to PX_\bullet \to X_\bullet$ and $PX_\bullet$ is contractible, so on realization, gives a map $|X_1| \to \Omega |X|$. For $\mc{C}$ a subtractive category, we may consider
\[
w \mc{C} = (\widetilde{S}_\bullet \mc{C})_1 \to P(\widetilde{S}_\bullet \mc{C}) \to \widetilde{S}_\bullet \mc{C}
\]
and obtain a map $|w \mc{C}| \to \Omega|\widetilde{S}_\bullet \mc{C}|$. This map will in general not be an equivalence, but upon applying $\widetilde{S}_\bullet$ on more time, it will be. To make this precise, we need the following proposition.

\begin{prop}\cite[Prop. 1.5.5]{waldhausen}
Let $\mc{A}, \mc{B}$ be SW-categories. Suppose $f: \mc{A} \to \mc{B}$ is exact. Then 
\[
w \widetilde{S}_\bullet \mc{B} \to w \widetilde{S}_\bullet \widetilde{S}_\bullet (f: \mc{A} \to \mc{B}) \to w \widetilde{S}_\bullet \widetilde{S}_\bullet \mc{A}
\]
is a fibration up to homotopy. 
\end{prop}
\begin{proof}
  This is exactly as in Waldhausen \cite[Prop. 1.5.5]{waldhausen}. 
\end{proof}

When $\mc{A} = \mc{B} = \mc{C}$ where $\mc{C}$ is a subtractive category, $w \widetilde{S}_\bullet \widetilde{S}_\bullet (f: \mc{C} \to \mc{C}) = P(w \widetilde{S}_\bullet \widetilde{S}_\bullet \mc{C})$, and we immediately obtain the following corollary. 

\begin{cor}
The sequence
\[
i \widetilde{S}_\bullet \mc{C} \to P (i \widetilde{S}_\bullet \widetilde{S}_\bullet \mc{C}) \to i \widetilde{S}_\bullet \widetilde{S}_\bullet \mc{C}
\]
is a fibration sequence up to homotopy, i.e. 
\[
|i S_\bullet \mc{C}| \simeq \Omega |i S_\bullet S_\bullet \mc{C}|
\]
\end{cor}

Finally, we have

\begin{thm}
Let $\mc{C}$ be a subtractive Waldhausen category. Then $\underline{K}(\mc{C})$ an infinite loop space. More precisely, $K(\mc{C})$ (see Defn. \ref{K_symmetric_spectrum}) is a quasi-fibrant symmetric spectrum. 
\end{thm}

\begin{rmk}
We will denote the associated delooped spectrum by $K(\mc{C})$. 
\end{rmk}

The main object of study is then obtained as a corollary:

\begin{cor}
Let $X$ be a scheme. There is a spectrum $K(\mbf{Var}_{/X})$ such that $\pi_0 K(\mbf{Var}_{/X})$ is the Grothendieck group of varieties over $X$.
\end{cor}

The next section shows we can do even better.

\subsection{Multiplicative Structure}

There is more structure to the category $\mbf{Var}_{/k}$ than we have used thus far, in particular, there is a cartesian product: given $k$-varieties $X, Y$ we can consider $X \times_k Y$. This much is used to produce the ring structure on $K_0 (\mbf{Var}_{/k})$. It will also produce a homotopy-coherent product structure on $K(\mbf{Var}_{/k})$, an $E_\infty$-ring structure.

Before going on, we introduce a useful construction in order to define products. We will follow Geisser-Hesselholt \cite{geisser_hesselholt} in defining products, and so follow them in defining an $S_\bullet$-construction appropriate to the task. The only modification is to consider $S^Q \mc{C}$ where $Q$ is a finite set. That is, instead of indexing on numbers, we index on finite sets. This serves to make the action by the symmetric group more transparent.

\begin{defn}
  Let $Q$ be a finite set. Consider positive integers $n_i$ indexed on $Q$, i.e. where $i \in Q$. Then $\widetilde{S}^Q_{n_1, \dots, n_{|Q|}} \mc{C}$ is a functor from the arrow category
  \[
  F: \widetilde{\text{Ar}}[n_1]\times \cdots  \times \widetilde{\text{Ar}}[n_{|Q|}] \to \mc{C}
  \]
  satisfying the same requirements as Defn.\ref{iterated_S_dot_construction}.  
\end{defn}

We now define

\begin{defn}
  Let $\mc{C}$ be an SW-category. The $K$-theory spectrum is given by
  \[
  K(\mc{C})(k) = |w_\bullet S^Q_\bullet \mc{C}|
  \]
  with $Q = \{1, \dots, k\}$. 
\end{defn}

We want to introduce a product structure on $K(\mc{C})$ from the product structure on $\mc{C}$. This can be done by exact analogy to the case of Waldhausen categories explained carefully in \cite{blumberg_mandell_koszul,geisser_hesselholt}. The structure necessary on $\mc{C}$ is that it be a permutative category (see, e.g. \cite{elmendorf_mandell} for an introduction to permutative categories) and that the product behave well with respect to subtractive structure (see Defn. \ref{bi_exact} below). Typically, however, we are given a symmetric monoidal structure, not a permutative structure on $\mc{C}$. Luckily, this presents no difficulty as symmetric monoidal categories can always be rigidified to \textit{equivalent} permutative categories \cite{isbell}. Since this procedure produces an equivalence of categories, the SW-structure may be carried along the equivalence.

The requirement that the product structure interact nicely with the subtractive structure amounts of the following requirement.

\begin{defn}\label{bi_exact}
  Let $\mc{C}$ be a permutative SW-category. Then a symmetric monoidal structure $\otimes: \mc{C} \times \mc{C} \to \mc{C}$ is \textbf{biexact} if
  \begin{enumerate}
  \item $X \times \emptyset$ and $\emptyset \times X$ are both $\emptyset$
  \item $X \times (-)$ and $(-) \times X$ are exact functors
  \item For $X \to Y$ and $X' \to Y'$ cofibrations the pushout-product
    \[
    X' \times Y \amalg_{X \times Y} X \times Y' \to X' \times Y'
    \]
    is a cofibration. 
  \end{enumerate}
\end{defn}

We then have (see \cite[p.40]{geisser_hesselholt})

\begin{defn}
  Let $\mc{C}$ be a permutative SW-category with bi-exact product. There is an induced product
  \[
  \widetilde{S}^Q_\bullet \mc{C} \times \widetilde{S}^{Q'}_\bullet \mc{C} \to \widetilde{S}^{Q \amalg Q'}_{\bullet} \mc{C}
  \]
  given by amalgamating the morphisms in the arrow categories. This gives a $\Sigma_m \times \Sigma_n$-equivariant map
  \[
  K(\mc{C})_m \times K(\mc{C})_n \to K(\mc{C})_{m+n} 
  \]
  which descends to
  \[
  K(\mc{C})_m \sma K(\mc{C})_n \to K(\mc{C})_{m+n}. 
  \]
\end{defn}

\begin{thm}\cite[Th 2.8]{blumberg_mandell_koszul}\cite[Prop. 6.1.1]{geisser_hesselholt}
Let $\mc{C}$ be a symmetric monoidal SW-category with $\otimes : \mc{C} \times \mc{C} \to \mc{C}$ biexact. Let $\overline{\mc{C}}$ denote the rigidification of $\mc{C}$. Then $K(\overline{\mc{C}}) \simeq K(\mc{C})$ is an $E_\infty$-ring spectrum.  
\end{thm}

Of course, we would like this result for $\mc{C} = \mbf{Var}_{/k}$. That means that we have to show that the cartesian product is biexact. It is clear that properties 1 and 2 of biexactness hold. Property 3 is the content of the proposition below. 

\begin{prop}
Let $X\hookrightarrow X'$ and $Y \hookrightarrow Y'$ be cofibrations of varieties. Then the pushout-product
\[
X \times Y' \amalg_{X \times Y} X' \times Y  \to X' \times Y'
\]
is a cofibration. 
\end{prop}

\begin{proof}
  The diagram
  \[
  \xymatrix{
    X\times Y \ar[d]\ar[r] & X' \times Y \ar[d]\\
    X \times Y' \ar[r] & X' \times Y'
  }
  \]
  is cartesian. Since we have verified the axioms for $\mbf{Var}_{/k}$, this means the pushout-product of this diagram is a cofibration. 
\end{proof}

\begin{cor}
The usual product induces a paing $\mbf{Var}_{/k} \times \mbf{Var}_{/k} \to \mbf{Var}_{/k}$ which descends to a product on $K(\mbf{Var}_{/k})$. Thus, $K(\mbf{Var}_{/k})$ is an $E_\infty$-ring spectrum. 
\end{cor}

\section{Maps out of $K(\mbf{Var}_{/k})$}

We come to the main point of the paper, which is to produce derived motivic measures, i.e. maps out of $K(\mbf{Var}_{/k})$. Even the structure of $K_0 (\mbf{Var}_{/k})$ is difficult to get one's hands on, and the progress made thus far has been through uses of motivic measures (see, e.g. \cite{larsen_lunts} for a beautiful example). In order to figure out the structure of the higher homotopy groups of $K(\mbf{Var}_{/k})$, it thus seems necessary to produce higher motivic measures. These maps take the form of spectrum maps $K(\mbf{Var}_{/k}) \to R$ where $R$ is any spectrum. Given a map of this form, we could take components to obtain $K_0 (\mbf{Var}_{/k}) \to \pi_0 R$ which is a classical motivic measure. As a first attempt at producing derived motivic meaures, we could thus ask for ones that lift known classical motivic measures. In this section, we will lift a number of classical motivic measures to such spectrum maps. This shows that in many known cases, classical motivic measures are the shadow of a much richer homotopical picture. 

Before we begin, an example illustrates the issue we will contend with:

\begin{example}
  Consider the category of complex varieties $\mbf{Var}_{/\mbf{C}}$. There is a motivic measure $K_0 (\mbf{Var}_{/\C})$ to $\mbf{Q}$-vector spaces obtained by taking compactly supported cohomology with $\mbf{Q}$-coefficients. For subtraction sequences $Z \hookrightarrow X \leftarrow X- Z$ this procedure is covariant with respect to closed inclusions and contravariant with respect to open inclusions and yields long exact sequences
  \[
\cdots \to  H^i_c (Z) \to H^i_c (X) \to H^i_c (X-Z) \to H^{i+1}_c (X) \to \cdots
\]
and so if we assign
\[
X \mapsto \chi(X) : = \sum [H^i_c (X;\mbf{Q})] \in K_0 (\mbf{Vect}_{\Q})
\]
we get a well-defined motivic measure.
\end{example}

To obtain a map $K(\mbf{Var}_{/k}) \to K(\mc{C})$ where $\mc{C}$ is a Waldhausen category, we need a map from the simplicial set $i \widetilde{S}_\bullet \mbf{Var}_{/k}$ into the simplicial sets $w S_\bullet \mc{C}$. In order to have such maps,  we will have to use functors that behave differently with respect to open and closed inclusions, because of the differences in vertical arrows in the respective $S_\bullet$-constructions. In fact, we'll have to deal with functors that are only \textit{really} functors on the subcategory of closed inclusions and subcategory of open inclusions, respectively.

The definition below is inspired by proper base change theorems in algebraic geometry. It was suggested to the author by Jesse Wolfson. He also pointed out that it is quite close to \cite[Defn. 3.3]{getzler}.

\begin{defn}\label{w_exact}
  Let $\mc{C}$ be an SW-ccategory and let $\mc{W}$ be a Waldhausen category. We define a \textbf{W-exact functor} from $\mc{C}$ to $\mc{W}$ to be a pair of functors $(F_!, F^!)$ such that
  \begin{enumerate}
  \item $F_!$ is a functor $F_!: \mbf{co}(\mc{C}) \to \mc{W}$. For $i$ a map we often denote $F_! (i)$ by $i_!$. 
  \item $F^!$ is a functor $F^!: \mbf{fib}(\mc{C})^{\text{op}} \to \mc{W}$. For $j$ a map we often denote $F^!(j)$ by $j^!$. 
  \item $F_! (X) = F^! (X)$ for $X \in \mc{C}$. We denote the common value by $F(X)$. 
  \item (\textbf{base change}) The cartesian diagram in $\mc{C}$
    \[
    \xymatrix{
      X \ar@{^{(}->}[d]_i \ar[r]^j_{\circ} & Z \ar@{^{(}->}[d]^{i'}\\
      Y \ar[r]^{\circ}_{j'} & W
      }
    \]
    produces a diagram
    \[
    \xymatrix{
      F(X)   \ar[d]_{i_!} & F(Z) \ar[d]^{(i')_!} \ar[l]_{j^!}\\
      F(Y)  & F(W) \ar[l]^{(j')^!}
    }
    \]
    and we require that the diagram commute, i.e. 
    \[
    i_! \circ j^! = (j')^! \circ (i')_! 
    \]
  \item (\textbf{excision}) For a subtraction sequence
    \[
    \xymatrix{
      X \ar@{^{(}->}[r]^i  & Y & Y - X \ar[l]^{\circ}_j
    }
    \]
    the induced sequence
    \[
    F (X) \xrightarrow{i_!} F (Y) \xrightarrow{j^!} F(Y - X)
    \]
    is a cofiber sequence in $\mc{W}$. 
  \end{enumerate}
For ease, we will write a W-exact functor as $(F_!, F^!): \mc{C} \to \mc{W}$ with the understanding that there is no underlying functor on the category $\mc{C}$. 
\end{defn}

We record the following consequence of the definition

\begin{prop}\label{pseudoexact}
  Given a W-exact functor, there is a map of simplicial sets $i \widetilde{S}_\bullet \mc{C} \to  w S_\bullet \mc{W}$ which induces a map of spectra \[K(\mc{C}) \to K(\mc{W}).\]
\end{prop}
\begin{proof}
  Consider an $n$-simplex $X \in \widetilde{S}_n \mc{C}$. Recall (Defn. \ref{s_dot}) that this means that $X$ is a functor $X: \widetilde{\operatorname{Ar}}[n] \to \mc{C}$ such that $X_{j, j} = \emptyset$ and every sub-diagram $X_{i,j} \to X_{i, k} \leftarrow X_{j, k}$ is a subtraction sequence.

  Apply $F_!$ to every cofibration and $F^!$ to every fibration in the diagram $X$. We note that by definition of $W$-exact functor, $F(X_{i,j}) \to F(X_{i,k})$ will be a cofibration in $\mc{W}$ and 
  \[
  F(X_{i,j}) \to F(X_{i,k}) \to F(X_{j,k})
  \]
  will be a cofiber sequence in $\mc{W}$. Thus, the image of $F$ lies in $S_\bullet \mc{W}$. 
\end{proof}

We also need a dual definition to prove maps \textit{from} a Waldhausen category to an SW-category. This situation seems to arise less commonly in practice, but will be useful below (Thm.\ref{splitting}). 

\begin{defn}
  Let $\mc{W}$ be a Waldhausen category ad $\mc{C}$ an SW-category. An \textbf{op-W-exact functor} is a pair of functors $(G_\ast, G^\ast)$ such that
  \begin{enumerate}
  \item $G_\ast$ is a functor $G_\ast : \text{co}(\mc{W}) \to \mc{C}$
  \item $G^\ast$ is a functor $G^\ast: \text{fib}(\mc{W})^{\text{op}} \to \mc{C}$
  \item For $X \in \mc{W}$, $G^\ast (X) = G_\ast (X)$. We refer to the common value as $G(X)$. 
  \item Given a diagram in $\mc{W}$
    \[
    \xymatrix{
      X \ar@{->>}[d]_j\ar[r]^i & Z \ar@{->>}[d]_{j'}\\
      Y \ar[r]_{i'} & W
    }
    \]
    where the horizontal maps are cofibrations and vertical maps are fibrations, we get the corresponding diagram in $\mc{C}$
    \[
    \xymatrix{
      G(X) \ar[r]^{i_\ast} & G(Z) \\
      G(Y) \ar[u]_{j^\ast} \ar[r]_{(i')_\ast} & G(W)\ar[u]_{(j')^\ast}
    }
    \]
    We require that the diagram commute, i.e. 
    \[
    i_\ast \circ j^\ast = (j')^\ast \circ (i')_\ast 
    \]
   \item Given a cofiber sequence in $\mc{W}$
     \[
     \xymatrix{
       X \ar[r]^i & Y \ar@{->>}[r]^j & Z 
     }
     \]
     we get a subtraction sequence in $\mc{C}$
     \[
       G(X) \xrightarrow{i_\ast}  G(Y) \xleftarrow{j^\ast}  G (Y - X)
     \]
  \end{enumerate}
\end{defn}

\begin{rmk}
Because of the rigidity of the category of varieties, these will be harder to produce in practice, in fact, the only example we know is the one below. 
\end{rmk}

By a proof entirely dual to Thm. \ref{pseudoexact}, we obtain

\begin{thm}\label{op_W_exact}
  Given a Waldhausen category $\mc{C}$, an SW-category $\mc{W}$ and an op-$W$-exact functor $(G_\ast, G^\ast)$ we get a map on $K$-theory spectra
  \[
  K(\mc{W}) \to K(\mc{C}). 
  \]
\end{thm}

In the subsections below, we will have occasion to use the category of pointed finite sets a number of times, so it worth defining before we get to work. 

\begin{defn}
Let $\mbf{FinSet}_+$ be the category of pointed finite sets. We choose a skeleton of it so that the objects are the pointed sets with $n$-elements $[\mbf{n}]_+$. Morphisms are maps preserving the basepoint, which we denote $\ast$. 
\end{defn}

The salient property of this category for us is the following celebrated theorem.

\begin{thm}[Barratt-Priddy-Quillen]
  Consider $\mbf{FinSet}_+$ as a Waldhausen category by defining cofibrations to be injective maps. Then
  \[
  K(\mbf{FinSet}_+) \simeq S
  \]
  where $S$ is the sphere spectrum. 
\end{thm}

Thus, $\mbf{FinSet}_+$ will be our category-level model of the sphere spectrum.

Below it will be necessary to view $\mbf{FinSet}_+$ as a Waldhausen category and also to understand some of its combinatorics.

First, we note that $\mbf{FinSet}_+$ can be made into a Waldhausen category by declaring that cofibrations are monomorphisms and weak equivalences are isomorphisms. We record the following definition for future use.

\begin{defn}
  A map $p: [\mbf{n}_1]_+ \to [\mbf{n}_2]_+$ will be said to be a \textbf{fibration} if it arises as a pushout
  \[
  \xymatrix{
    [\mbf{n}_0]_+ \ar[r]^i\ar[d] & [\mbf{n}_1]_+\ar[d]_p\\
    \ast \ar[r] & [\mbf{n}_2]_+
  }
  \]
  where $i$ is a cofibration. More concretely, $p$ is a fibration if it is surjective and for $i \in [\mbf{n}_2]_+$, $p^{-1}(i)$ has one element. 
\end{defn}

We define two flavors of wrong way maps in $\mbf{FinSet}_+$. 

\begin{defn}\label{cofibration_backward_finset}
Let $f: [\mbf{n}_1]_+ \to [\mbf{n}_2]_+$ be a monomorphism in $\mbf{FinSet}_+$. We define $f^\ast: [\mbf{n}_2]_+ \to [\mbf{n}_1]_+$ by mapping the corange to the basepoint and each $i \in \im(f)$ to $f^{-1}(i)$. 
\end{defn}

\begin{defn}\label{fibration_backward_finset}
Let $p: [\mbf{n}_1]_+ \to [\mbf{n}_2]_+$ be a fibration. We define $p^\ast$ as follows. For $i \in \im(p)$, define $p^\ast (i) = p^{-1}(i)$ and then map the basepoint to the basepoint. 
\end{defn}

We now consider commutative diagrams
\[
\xymatrix{
  [\mbf{n}_1]_+ \ar@{->>}[d]_{p_1} \ar@{^{(}->}[r]^{i_1}  & [\mbf{n}_2]_+\ar@{->>}[d]^{p_2}\\
  [\mbf{n}_3]_+ \ar@{^{(}->}[r]_{i_2} & [\mbf{n}_4]_+
}
\]
Commutativity in this case means that $p^{-1}_1 (\ast) = i^{-1}_1 (p^{-1}_2 (\ast))$ and that for $i \in [\mbf{n}_4]_+$, $(i_2 \circ p_1)^{-1}(i) = (p_2 \circ i_1)^{-1} (i)$.

This observation has the following simple, but useful, consequence.

\begin{lem}\label{backwards_commute_finset}
  Given a commutative diagram as above, the following also commutes
  \[
  \xymatrix{
    [\mbf{n}_1]_+ \ar[r]^{i_1} & [\mbf{n}_2]\\
    [\mbf{n}_3]_+ \ar[u]^{p^\ast_1}\ar[r]_{i_2} & [\mbf{n}_4]_+ \ar[u]_{p^\ast_2}
  }
  \]
\end{lem}

\subsection{The Unit Map}

Since it is a spectrum, $K(\mbf{Var}_{/k})$ naturally has a unit map from the sphere spectrum $S \to K(\mbf{Var}_{/k})$. It will be useful for us to have a model for this map. When working with $K$-theoretic functors, finite pointed sets are always a proxy for the sphere spectrum, by Barrat-Priddy-Quillen. We construct functors out of this category to model maps out of the sphere spectrum. 

\begin{defn}
  We define an op-W-exact functor $(G_*,G^*): \mbf{FinSet}_+ \to \mbf{Var}_{/k}$ as follows.

  \begin{enumerate}
  \item $G_\ast: \mbf{FinSet}_+ \to \mbf{Var}_{/k}$ is defined on objects by
    \[
    G_\ast ([\mbf{n}_1]) = \coprod^{n_1}_{i=0} \Spec (k). 
    \]
    One cofibrations, i.e. inclusions it is defined by the corresponding inclusions of of $\Spec (k)$s. On fibrations, it is defined by the corresponding fold maps.

  \item $G^\ast: \mbf{FinSet}_+ \to \mbf{Var}_{/k}$ is defined by objects as above. Given a cofibration $i: [\mbf{n}_1]_+ \to [\mbf{n}_2]_+$, we define $G^\ast$ to be $G_\ast (i^\ast)$ with $i^\ast$ defined as in \ref{cofibration_backward_finset}. Given a fibration $p: [\mbf{n}_1]_+ \to [\mbf{n}_2]_+$ we define $G^\ast (p)$ to be $G_\ast (p^\ast)$ with $p^\ast$ defined as in \ref{fibration_backward_finset}.
  \end{enumerate}
\end{defn}

\begin{prop}
The map above is in fact op-W-exact. 
\end{prop}
\begin{proof}
  The first three conditions are trivial. To check the 4th, we consider a diagram in $\mbf{FinSet}_+$
  \[
  \xymatrix{
    [\mbf{n}_1]_+ \ar@{->>}[d]_j \ar[r]^i & [\mbf{n}_2]_+ \ar@{->>}[d]^{j'}\\
    [\mbf{n}_3]_+ \ar[r]_{i'} & [\mbf{n}_4]_+
    }
  \]
  This induces a diagram of varieties
  \[
  \xymatrix{
  G([\mbf{n}_1]_+) \ar[r]^{G_\ast (i)}  & G([\mbf{n}_2]_+)\\
  G([\mbf{n}_3]_+) \ar[r]_{G_\ast (i')}\ar[u]^{G^\ast (j)}   & G([\mbf{n}_4]_+) \ar[u]_{G^\ast (j')}
  }
  \]
  We now check that the two maps we need to agree in fact agree. That is, we need
  \[
G_\ast (i) \circ G^\ast (j) = G^\ast (j') \circ G_\ast (i')
  \]

  However, this is the content of Lem \ref{backwards_commute_finset}. 
 
\end{proof}

\begin{cor}
The op-W-exact functors descend to a map of spectra $S \to K(\mbf{Var}_{/k})$. 
\end{cor}

\begin{rmk}
  It is not hard to see that we get an $E_\infty$-map $S \to \mbf{Var}_{/k}$, but this will not be needed. 
\end{rmk}

\subsection{Point Counting}

One of the fundamental goals of algebraic geometry is to systematically count points on algebraic varieties over finite fields. This procedure would take an algebraic variety over a finite field $k$ and return the number of $k$-points $|X(k)|$. Such a procedure behaves well with respect to subtracting varieties, and so it descends to a motivic measure $K_0 (\mbf{Var}_{/k}) \to \Z$. This is the first motivic measure that we will lift. 

\begin{defn}
  Define a W-exact functor $(-(k)_!, -(k)^!): \mbf{Var}_{/k} \to \mbf{FinSet}_+$ as follows.  On objects, we define the functor to be $X(k)_+$, the set of $k$-points of $X$ with a  disjoint basepoint added. We assign a linear order to the points, once and for all. We assign closed inclusions $Z \hookrightarrow X$ to be the obvious inclusion $Z(k)_+ \to X(k)_+$. For open inclusions, $X \xleftarrow{\circ} Y$, define $X(k)_+ \to Y(k)_+$ by restriction coupled with the requirement that if $p \in X(k)$, but $p \notin Y(k)$ then $p$ maps to the basepoint. 
\end{defn}

\begin{prop}
The functors $-(k)^!$ and $-(k)_!$ assemble into a W-exact functor $\mbf{Var}_{/k} \to \mbf{FinSet}_+$. 
\end{prop}
\begin{proof}
  We need to verify the conditions of Def. \ref{pseudoexact}. Suppose we have a commutative square
  \[
  \xymatrix{
    X \ar@{^{(}->}[r]^j\ar[d]^{\circ}_i & Z\ar[d]^{i'}_{\circ} \\
    Y \ar@{^{(}->}[r]_{j'} & W 
  }
  \]
  where $j, j'$ are closed and $i$,$i'$ are open. Then we get an induced square in $\mbf{FinSet}_+$
  \[
  \xymatrix{
  X(k)_+ \ar[r]^{j_!}  & Z(k)_+ \\
  Y(k)_+ \ar[r]_{(j')_!} \ar[u]^{i^!} & W(k)_+\ar[u]_{(i')_!}
  }
  \]
  which we would like to be commutative. However, this is a consequence of Lem. \ref{backwards_commute_finset}. 
\end{proof}

Note that if we have two $k$-varieties $X, Y$ then the number of $k$-points in $X \times_k Y$ is the product of the number of $k$-points in $X$ and $Y$. This product can be made functorial. 

\begin{thm}
  There is a map of spectra $K(\mbf{Var}_{/k}) \to S$ 
\end{thm}
\begin{proof}
By the previous proposition, we have a $W$-exact functor $\mbf{Var}_{/k} \to \mbf{FinSet}_+$. By Thm. \ref{pseudoexact} this induces a map of spectra $K(\mbf{Var}_{/k}) \to K(\mbf{FinSet}_+)$.  Barrat-Priddy-Quillen finishes the proof. 
\end{proof}

\begin{rmk}
This too is a map of $E_\infty$-ring spectra. 
\end{rmk}

\begin{prop}\label{splitting}
The composition of the point-cointing map with the unit map is the identity, thus the sphere spectrum splits off of $K(\mbf{Var}_{/k})$ and we may write $K(\mbf{Var}_{/k}) \simeq S \vee \widetilde{K}(\mbf{Var}_{/k})$. 
\end{prop}
\begin{proof}
  We consider the compostion of W-exact and op-W-exact functors
  \[
  \mbf{FinSet}_+ \xrightarrow{(G_\ast, G^\ast)} \mbf{Var}_{/k} \xrightarrow{(-(k)_!, -(k)^!)} \mbf{FinSet}_+.
  \]
  It is easy to see that this is the identity. The first map is op-W-exact and the second is $W$-exact. Thus, by Thm. \ref{pseudoexact} and Thm. \ref{op_W_exact} we obtain
  \[
  S \to K(\mbf{Var}_{/k}) \to S. 
  \]
\end{proof}

\subsection{Map to Waldhausen A-Theory}

Throughout this subsection we work over the base field $\C$. In this case varieties may be considered as topological spaces. However, there is already a $K$-theory of topological spaces, namely, Waldhausen's $A$-theory \cite[p.383]{waldhausen}. We produce a map $K(\mbf{Var}_{/\C}) \to A(\ast)$ relating these two $K$-theories. 

First, we recall the definition of Waldhausen's $A(\ast)$.

\begin{defn}\cite[p.379]{waldhausen}
Let $\mc{R}^{hf}_{/\ast}$ be the Waldhausen category of homotopy finite retractive spaces. These are spaces homotopy equivalent to a finite complex, equipped with cofibrations given by the homotopy extension property and weak equivalences the usual weak equivalences. 
\end{defn}

\begin{defn}
  The \textbf{Waldhausen $A$-theory} of a point is
  \[
  A(\ast) = \Omega |w S_\bullet R^{hf}_{/\ast}|
  \]
\end{defn}

In order to produce a map from $K(\mbf{Var}_{/\C})$ to $A(\ast)$, We need to produce a W-exact map $\mbf{Var}_{/\C} \to \mc{R}^{hf}_{/\ast}$. First, there is a forgetful functor $\mbf{Var}_{/\C} \to \mbf{Top}$ given by considering the smooth variety as a topological space.

The following result is folklore \cite{cisinski}

\begin{prop}\label{sep_fin_type}
Consider $X$ a separated, finite-type, complex scheme. If we consider it as a topological space and consider the one point compactification $X^+$, then $X^+$ is homotopy equivalent to a finite CW-complex. 
\end{prop}

\begin{prop}
The one-point compactification functor $((-)^+_!, (-)^{+,!}):  \mbf{Var}_{/\C} \to \mc{R}^{hf}_{/\ast}$ is W-exact. 
\end{prop}
\begin{proof}
One point compactification is covariant with respect to proper maps between topological spaces and contravariant with respect to open inclusions. The necessary diagrams obviously commute. 
\end{proof}

We thus obtain 

\begin{thm}
  There is a map of spectra
  \[
  K(\mbf{Var}_{/\C}) \to A(\ast)
  \]
\end{thm}

\begin{rmk}
The homotopy groups and homotopy type of $A(\ast)$ have recently been computed \cite{blumberg_mandell_A_point,blumberg_mandell_A_point_II}. It would be very interesting to know what parts of this are picked up by $K(\mbf{Var}_{/\C})$. 
\end{rmk}

\begin{rmk}
  Using trace methods there is a map $A(\ast) \to S$, and thus a composition
  \[
  K(\mbf{Var}_{/\C}) \to A(\ast) \to S. 
  \]
  This is likely the analgoue of point-counting or the Euler characteristic.
\end{rmk}

\begin{rmk}
The reliance on Prop. \ref{sep_fin_type} is somewhat unsatisfactory. However, there are much cleaner, more ``motivic'',  ways of producing this map, as suggested to the author by Denis-Charles Cisinski \cite{cisinski}. We will pursue these in future work. 
\end{rmk}

Granted the above map, we can also obtain a map to any $K(R)$ for $R$ a ring or ring spectrum. The $A$-theory of a point is equivalent to the spectrum $K(S)$. There is a functor $\mbf{Var}_{/\C}$ to spectra (i.e. $S$-modules) specified by $X \mapsto \Sigma^\infty X(\C)_+$. By smashing with any ring spectrum $R$ we obtain a functor $\mbf{Var}_{/\C} \to \operatorname{Mod}_{R}$. In the case when $R$ is an Eilenberg-MacLane spectrum $HA$, this is equivalent to considering the compactly-supported cohomology of $X$ with coefficients in $R$.

\section{Conjectures and Future Work}

This paper has set up a model for investigating $K(\mbf{Var}_{/k})$. There are of course further points to investigate. Not only are there many more derived motivic measures, but one may wonder about the relationship with other aspects of $K_0 (\mbf{Var}_{/k})$, for example, whether motivic integration could be lifted.

Let us briefly discuss a conjectural motivic measure. When looking for a motivic measure, we of course have to produce W-exact functors, and thus need functors with certain specific variance properties. We consider one such functor presently. 

Let $X$ be a Noetherian scheme. Quillen defines $K'(X)$ to be $K(\mbf{Coh}(X))$, that is he defines it to be the $K$-theory of the abelian category of coherent sheaves on $X$ \cite{quillen}. He also proves the following proposition

\begin{prop}\cite[3.1]{quillen}
  Let $X \hookrightarrow Y$ be a closed immersion. Then there is a cofibration sequence of spectra
  \[
  K'(X) \to K'(Y) \to K'(Y-X)
  \]
\end{prop}

This means that $K'(-)$ is exactly the sort of functor that we need. It is covariant with respect to closed inclusions, and contravariant with respect to open inclusions. It thus gives us a W-exact functor $K': \mbf{Var}_{/k} \to \mbf{Sp}$ where the latter is the category of spectra considered as a Waldhausen category via its model structure.  Furthermore, every $K$-theory spectrum $K'(X)$ is a $K(S)$-module. Thus the $K'$ functor is actually an exact functor
\[
K': \mbf{Var}_{/k} \to \mbf{Mod}_{K(S)}
\]
where the latter denotes the modules over the $E_\infty$-ring $K(S)$. We would like this to produce a map on $K$-theory. However, by the Eilenberg swindle, the $K$-theory of $\mbf{Mod}_{K(S)}$ vanishes. In order to get a map $K(\mbf{Var}_{/k}) \to K(K(S))$ we require that $K'$ land in \textit{compact} (or perhaps dualizable) $K(S)$-modules. To put this more succinctly, we have two conjectures, the former implied by the latter.

\begin{conj}
  There is a map of ring spectra
  \[
  K(\mbf{Var}_{/k}) \to K(K(S))
  \]
\end{conj}

\begin{conj}
Let $X$ be a smooth scheme. Then $K(X)$ is compact or dualizable as a $K(S)$-module.   
\end{conj}

\begin{rmk}
  When $X$ is a $k$-variety, $K'(X)$ is also a $K(k)$-module. It is also possible that $K'(X)$ could be compact as a $K(k)$-module, in which case we would have a map
  \[
  K(\mbf{Var}_{/k}) \to K(K(k))
  \]
\end{rmk}

\section{Appendix: Simplicial Homotopy}

In this appendix we present a few diagrams to aid in intuition with the simplicial homotopy produced in the additivity theorem. The simplex $h_3(e)$ where $e$ is a 5-simplex looks like

\[
\xymatrix@C=.5cm@R=.5cm{
  A_0 \ar@{^{(}->}[r]& A_1 \ar@{^{(}->}[r] & A_2 \ar@{^{(}->}[r] & A_3 \ar@{^{(}->}[r] & S_0 \ar@{=}[r]& S_0 \ar@{=}[r] & S_0\\
  & A_{1, 1} \ar@{^{(}->}[r]\ar[u]^{\circ} & A_{1,2} \ar[u]^{\circ} \ar@{^{(}->}[r] & A_{1,3} \ar@{^{(}->}[r]  \ar[u]^{\circ} & S_0 - A_1 \ar@{=}[r] \ar[u]^{\circ} & S_0 - A_1 \ar@{=}[r] \ar[u]^{\circ} & S_0 - A_1  \ar[u]^{\circ} \\
  & & A_{2,2}  \ar[u]^{\circ} \ar@{^{(}->}[r]& A_{2,3}   \ar[u]^{\circ}\ar@{^{(}->}[r]& S_0 - A_2  \ar[u]^{\circ}\ar@{=}[r] & S_0 - A_2  \ar[u]^{\circ}\ar@{=}[r] & S_0 - A_2 \ar[u]^{\circ} \\
  & & & A_{3,3} \ar[u]^{\circ} \ar@{^{(}->}[r] & S_0 - A_3  \ar[u]^{\circ}\ar@{=}[r] & S_0 - A_3  \ar[u]^{\circ} \ar@{=}[r]& S_0 - A_3 \ar[u]^{\circ} & \\
  & & & &\emptyset & \emptyset & \emptyset \\
  & & & & & \emptyset& \emptyset \\
  & & & & &          & \emptyset
}
\]

The more important part of the simplicial homotopy is the $h^C_i(e)$ simplex. The ppicture below is of $h_3 (e)$ when $e$ is a 5-simplex. For compactness we write $S_{i,0} := S_0 - A_i$. 

\[
\xymatrix@C=.2cm @R=.5cm{
  C_0 \ar@{^{(}->}[r]& C_1 \ar@{^{(}->}[r] & C_2 \ar@{^{(}->}[r] & C_3 \ar@{^{(}->}[r] & C_3 \amalg_{A_3} S_0 \ar@{^{(}->}[r] & C_4 \amalg_{A_4} S_0 \ar@{^{(}->}[r] & C_5 \amalg_{A_5} S_0\\
  & C_{1,1} \ar[u]^\circ \ar@{^{(}->}[r] & C_{1,2}  \ar[u]^\circ \ar@{^{(}->}[r] & C_{1,3} \ar[u]^\circ \ar@{^{(}->}[r] & C_{1,3} \amalg_{A_{1,3}} S_{1,0}  \ar[u]^\circ \ar@{^{(}->}[r] & C_{1,4} \amalg_{A_{1,4}} S_{1,0}  \ar[u]^\circ \ar@{^{(}->}[r] & C_{1,5} \amalg_{A_{1,5}} S_{1,0} \ar[u]^\circ\\
  & & C_{2,2} \ar[u]^\circ \ar@{^{(}->}[r] & C_{2,3}  \ar[u]^\circ \ar@{^{(}->}[r] & C_{2,3} \amalg_{A_{2,3}} S_{2,0}  \ar[u]^\circ \ar@{^{(}->}[r] & C_{2,4} \amalg_{A_{2,4}} S_{2,0}  \ar[u]^\circ \ar@{^{(}->}[r] & C_{2,5} \amalg_{A_{2,5}} S_{2,0} \ar[u]^\circ \\
  & & & C_{3,3}  \ar[u]^\circ \ar@{^{(}->}[r] & S_{3,0}  \ar[u]^\circ \ar@{^{(}->}[r] & C_{3,4} \amalg_{A_{3,4}} S_{3,0}  \ar[u]^\circ \ar@{^{(}->}[r] & C_{3,5} \amalg_{A_{3,5}} S_{3,0}  \ar[u]^\circ \\
  & & & &\emptyset & B_{3,4}  \ar[u]^\circ \ar@{^{(}->}[r] & B_{3,5} \ar[u]^\circ \\
    & & & & & B_{4,4} \ar[u]^\circ\ar@{^{(}->}[r] & B_{4,5} \ar[u]^\circ \\
  & & & & & & B_{5,5} \ar[u]^\circ \\
}
\]

For completeness, we include the $h_i(e)^B$ in this case as well.

\[
\xymatrix@C=.5cm@R=.5cm{
  B_0\ar@{^{(}->}[r]& B_1\ar@{^{(}->}[r] & B_2\ar@{^{(}->}[r] & B_3\ar@{=}[r] & B_3 \ar@{^{(}->}[r]& B_4\ar@{^{(}->}[r] & B_5\\
  & B_{1,1}\ar@{^{(}->}[r]\ar[u]^{\circ} & B_{1,2}\ar@{^{(}->}[r]\ar[u]^{\circ} & B_{1,3}\ar@{=}[r]\ar[u]^{\circ} & B_{1,3}\ar@{^{(}->}[r]\ar[u]^{\circ} & B_{1,4}\ar@{^{(}->}[r]\ar[u]^{\circ} & B_{1,5}\ar[u]^{\circ}\\
  & & B_{2,2}\ar@{^{(}->}[r]\ar[u]^{\circ} & B_{2,3}\ar@{=}[r]\ar[u]^{\circ} & B_{2,3}\ar@{^{(}->}[r]\ar[u]^{\circ} & B_{2,4}\ar@{^{(}->}[r]\ar[u]^{\circ} & B_{2,5}\ar[u]^{\circ}\\
  & & & B_{3,3}\ar@{=}[r]\ar[u]^{\circ} & B_{3,3} \ar@{^{(}->}[r]\ar[u]^{\circ}& B_{3,4} \ar@{^{(}->}[r]\ar[u]^{\circ}& B_{3,5}\ar[u]^{\circ}\\
  & & & & B_{3,3}\ar@{^{(}->}[r]\ar[u]^{\circ} & B_{3,4}\ar@{^{(}->}[r]\ar[u]^{\circ} & B_{3,5}\ar[u]^{\circ}\\
  & & & & &B_{4,4}\ar@{^{(}->}[r]\ar[u]^{\circ} & B_{4,5}\ar[u]^{\circ} \\
  & & & & & & B_{5,5}\ar[u]^{\circ}   
}
\]

\bibliographystyle{amsplain}
\bibliography{kvarbib}

\end{document}